\documentclass[12pt,a4paper,oneside,reqno]{amsart}
\usepackage{fancyhdr}
\usepackage{tikz}
\usepackage{pgfplots}
\pgfplotsset{compat=1.18}
\usepackage{amssymb,mathrsfs,amsmath,amsthm}
\usepackage[numbers,sort&compress]{natbib}
\usepackage{float}
\usepackage{authblk}
\usepackage{url}
\usepackage[numbers]{natbib}
\usepackage{setspace}
\usepackage{mathtools}
\usepackage{comment}
\usepackage[inline]{enumitem}
\usepackage[
  colorlinks=true,
  linkcolor=blue,
  citecolor=blue,
  urlcolor=blue,
  pagebackref,
  pdfauthor={Your Name},
  pdftitle={Your Document Title},
  pdfkeywords={keyword1, keyword2, keyword3}
]{hyperref}
\usepackage{booktabs}
\usepackage{comment}
\usepackage{hyperref}
\usepackage[margin=1in]{geometry} 

\newtheorem{theoremA}{Theorem}

\newtheorem{thm}{Theorem}[section]
\newtheorem{cor}[thm]{Corollary}
\newtheorem{lem}[thm]{Lemma}
\newtheorem{prop}[thm]{Proposition}
\theoremstyle{definition}

\newtheoremstyle{boldremark}
  {10pt}  
  {10pt}   
  {}      
  {}       
  {\bfseries} 
  {.}      
  { }     
  {}  
\theoremstyle{boldremark}
\newtheorem{Remark}[thm]{\textbf{Remark}}
\numberwithin{equation}{section}
\newenvironment{mathclass}
  {Mathematics Subject Classification (2010):}
 
\newcommand{\R}{\mathbb{R}}
\newcommand{\C}{\mathbb{C}}

\renewcommand{\keywords}[1]{%
  \par\noindent
  Keywords: #1
  \par
}

\def\R {\mathbb{R}}

\newcommand{\be}{\begin{equation}}
\newcommand{\ee}{\end{equation}}
\newcommand{\bea}{\begin{eqnarray}}
\newcommand{\eea}{\end{eqnarray}}
\newcommand{\Bea}{\begin{eqnarray*}}
\newcommand{\Eea}{\end{eqnarray*}}
\newcommand{\bt}{\begin{Theorem}}
\newcommand{\et}{\end{Theorem}}

\newcommand{\bpr}{\begin{Proposition}}
\newcommand{\epr}{\end{Proposition}}

\newcommand{\bl}{\begin{Lemma}}
\newcommand{\el}{\end{Lemma}}
\newcommand{\bi}{\begin{itemize}}
\newcommand{\ei}{\end{itemize}}

\newtheorem{Definition}{Definition}[section]
\newtheorem{Theorem}[Definition]{Theorem}
\newtheorem{Lemma}[Definition]{Lemma}
\newtheorem{Proposition}[Definition]{Proposition}

\title
{On  1D mass subcritical  nonlinear Schr\"odinger and Hartree equations  in
modulation spaces  $M^{p, p'} \ (p<2)$ }
\author{Divyang G. Bhimani}
\address{Divyang G. Bhimani\\Department of Mathematics\\Indian Institute of Science Education and Research\\ Pune 411008\\India}
\email{divyang.bhimani@iiserpune.ac.in}

\author{Diksha Dhingra}
\address{Diksha Dhingra \\Department of Mathematics\\ Indian Institute of Technology\\ Indore, 452020\\ India}
\email{dikshadd1996@gmail.com}

\author{Vijay Kumar Sohani}
\address{Vijay Kumar Sohani\\ Department of Mathematics\\ Indian Institute of Technology\\ Indore 452020\\ India}
\email{vsohani@iiti.ac.in}

\fancypagestyle{plain}{\fancyhf{} }

\begin{document}
\date{}
\maketitle{}
\begin{center}
DIVYANG G. BHIMANI, DIKSHA DHINGRA, VIJAY KUMAR SOHANI
\end{center}
\begin{abstract}
We establish well-posedness theory for the 1D mass-subcritical nonlinear Schr\"odinger equation (NLS) having power-type nonlinearity $|u|^{\alpha-1}u$ in a certain modulation spaces 
$M^{p,p'}(\mathbb R),$ where $p'$ is a H\"older conjugate of $p$, with $4/3<p<2$ and $p$ sufficiently close to $2$. Modulation spaces have been successfully applied in understanding the dynamics of NLS near the Sobolev scaling critical regularity. In fact, despite cubic NLS is ill-posed in $H^s$ for $s<-1/2$;  our analysis reveals that it experiences well-posedness in modulation spaces for a  Cauchy data in $(H^s\setminus L^2) \cap M^{p,p'}$. The proof adopts two different approaches to establish local well-posedness for $\alpha \in (1,5)$: one exploits generalised Strichartz estimates in Fourier-Lebesgue and Lebesgue spaces; the other implements Bourgain’s high-low decomposition (BHLD) method in the modulation space setting. The local solution via the (BHLD) method can be extended to global-in-time, but with a certain loss of regularity.  
We could combine these effectively and establish global well-posedness in $M^{p,p'}$ with the persistence of regularity for $1<\alpha \leq 10/3$. This is the first global result in $M^{p,p'}$
which establishes the persistence of regularity. Similar results are also established for the Hartree equations.\\

\end{abstract}
\footnotetext{\begin{mathclass} Primary 35Q55, 35Q60; Secondary 35A01.
\end{mathclass}
}
\section{Introduction}
\subsection{Motivation and Background}\label{mb}
On the  real  line, we  consider the following nonlinear Schr\"odinger  equation 
\begin{equation}\tag{NLS}\label{NLS}
\begin{cases}
    iu_t +u_{xx}\pm (|u|^{\alpha-1}u)=0\\ u(\cdot,0)=u_0
\end{cases}, \quad (1< \alpha<5)  \end{equation}  and  Hartree equation
\begin{equation}\tag{H}\label{NLSH}
\begin{cases}
    iu_t + u_{xx} \pm (|\cdot|^{-\nu} \ast |u|^{2})u=0\\ u(\cdot,0)=u_0
\end{cases}, \quad (0<\nu<1).
\end{equation}
Here $ (x , t ) \in \mathbb R \times \mathbb R $ and $ u(x,t) \in \mathbb C$. 
Formally, the solutions of   \eqref{NLS} and \eqref{NLSH} enjoys the conservation of mass:
\begin{equation}\label{mass}
M[u(t)] = \int_{\R} |u(x, t)|^2 \, dx = M[u_{0}].
\end{equation}
\par{The main goal of this article is to establish well-posedness theory for \eqref{NLS} and \eqref{NLSH} in the setting of low regularity spaces  (modulation spaces $M^{p,p'}, p<2$). See Section \ref{sssm} below for a precise definition. The well-posedness theory for \eqref{NLS} and \eqref{NLSH} in $L^2(\mathbb{R})$ is due to Y. Tsutsumi \cite{YTsutsumi} and Chadam-Glassey \cite{ChadamGlassey}, respectively.  A variety of physically interesting dynamics emerge where the standard \( L^2 \)-theory no longer applies, and low regularity is more desirable.  For instance, the LIA (localised induction approximation)  model for a vortex filament appears in the theory of superfluid turbulence. See \cite[Section 1]{vargas2001global} for details. In fact,  taking this into account,  Vargas and Vega \cite{vargas2001global} and Gr\"unrock \cite{Grunrock} proved local and global well-posedness for cubic ($\alpha=3$) \eqref{NLS} with Cauchy data having infinite $L^2-$norm and Fourier-Lebesgue spaces, respectively. We also refer to the influential papers \cite{MiaoJPDE, MiaoJFA} by Miao, Xu, and Zhao for the Hartree equation. }

On the other hand,  up to now, we cannot establish a well-posedness theory for  \eqref{NLS} and \eqref{NLSH}
with Cauchy data in  $L^p-$spaces except for $p=2$ mainly  because $U(t)=e^{it\partial_x^2}:L^p\to L^p $ iff $p=2$. However,   Zhou \cite{ZhouLp} and Hyakuna \cite{HyakunaJFA, HyakunaIUMJ} 
established  well-posedness in $L^p-$spaces in the sense of ``twisted regularity". To explain this, we recall that the corresponding integral equation to \eqref{NLS} is 
\begin{equation}\label{iq}
    u(t)= U(t)u_0 \pm i \int_0^t U(t-s) \left[|u(s)|^{\alpha-1} u(s) \right]ds.
\end{equation}
The key idea in \cite{ZhouLp,HyakunaJFA, HyakunaIUMJ} is to introduce the twisted variable $v(t)=U(-t)u(t)$ to transform  \eqref{iq} as
\begin{eqnarray}\label{tiq}
    v(t)= u_0 \pm i \int_0^t U(-s)\left[ |U(s) v(s)|^{\alpha-1} U(s)v(s) \right]ds.
\end{eqnarray}
Then, they successfully established well-posedness for the twisted integral \eqref{tiq} in $ L^p$-spaces for $1\leq p< 2$. We refer to  \cite{BhimaniAHP, HyakunaIUMJ, ZhouLp, HyakunaJFA}  for further progress in this direction. However, it is strongly desirable to establish a well-posedness theory of \eqref{iq} (i.e. without any twisting) in many situations. This inspires us to ask the following questions:  
\begin{itemize}
    \item[--] Can we go beyond $L^p-$spaces to establish well-posedness theory for \eqref{NLS}? Can we have a Cauchy data class that contains $ L^p$-spaces, and still expect well-posedness theory for \eqref{NLS}?
\end{itemize}
We recall  \cite{Sugimoto1}  sharp embedding between $L^p$ and modulation spaces:
\begin{equation}\label{dsembeddings}
    L^{p}(\mathbb R) \hookrightarrow M^{p,p'} (\mathbb R) \quad  \text{for} \quad 1\leq p<2
\end{equation}
where $1/p+1/p'=1.$ In view of  embedding  \eqref{dsembeddings} and taking the Cauchy data in   $M^{p,p'}$, we could  answer the aforementioned  question  satisfactorily. See  Theorem \ref{betterregularityNLS} - Corollary \ref{globalprnls} below. 

 We now turn our attention to connecting our analysis with  Sobolev spaces $H^s$. Recently,  Griffiths, Killip and Visan   \cite{BenjaminSharpNLS} showed that cubic \eqref{NLS} is globally well-posed in \(H^s(\mathbb R)\) for $s>-1/2$ (the scaling critical index); while Oh \cite{OhHsIll} proved it is strongly ill-posed (norm-inflation) in $H^s (\mathbb R)$ for $s\leq -1/2$. Taking into account \eqref{srp1}, we have
\begin{eqnarray}\label{dsei}
   M^{p,p'}(\mathbb R) \hookrightarrow  H^{s}(\mathbb R) 
\quad \text{for} \quad  s < \frac{1}{2} - \frac{1}{p},  \ \  p<2,
\end{eqnarray}
 and there exist $u_0\in M^{p,p'} (\mathbb R)   (1\leq p<2)$ such that $u_0\notin L^2(\mathbb R)$. Thus we have $$u_0\in (H^s \setminus L^2)(\mathbb R) \cap M^{p,p'} (\mathbb R). $$  
 Surprisingly, this  reveals that although  cubic \eqref{NLS}  experiences strong ill-posedness (norm inflation) in $H^s$ for $s\leq -1/2;$ there is data in $(H^s\setminus L^2) \cap M^{p,p'}$ for which \eqref{NLS} is well-posed in $M^{p,p'}$. See Theorem \ref{betterregularityNLS} - Corollary \ref{globalprnls} below.
\subsubsection{\textbf{Modulation spaces}}\label{sssm}
\par{In the past 20 years, modulation spaces have been very successful in understanding the dynamics of  \eqref{NLS} and \eqref{NLSH} with rough data (low regularity) and near scaling critical regularity. See \cite{HyakunaIUMJ,bhimani2016cauchy,wang2011harmonic, Ruz4NLS, oh2018global, Kasso2009, KassoBook, Wang2006,Bhimani2016, BhimaniJFA2, BhimaniNorm, BhimaniHartree-Fock,DiracHartree}. Historically, these spaces originated in the pioneering work of H. Feichtinger \cite{Feih83} during the early 1980s, and they are now omnipresent in many applications \cite{Fei2006}. To define these spaces, we set some notations.}
 \par{Let $\rho: \mathbb R \to [0,1]$  be  a smooth function satisfying   $\rho(\xi)= 1 \  \text{if} \ \ |\xi|\leq \frac{1}{2} $ and $\rho(\xi)=
0 \  \text{if} \ \ |\xi| \geq  1$. Let  $\rho_k$ be a translation of $\rho,$ that is,
$ \rho_k(\xi)= \rho(\xi -k) \ (k \in \mathbb Z).$
Denote 
$\sigma_{k}(\xi)= \frac{\rho_{k}(\xi)}{\sum_{l\in\mathbb Z}\rho_{l}(\xi)} \ (k \in \mathbb Z).$  This family $\{\sigma_k\}$ of smooth functions gives a bounded uniform partition of unity. 
The frequency-uniform decomposition operators can be  defined by 
$\square_k = \mathcal{F}^{-1} \sigma_k \mathcal{F} \quad(k \in \mathbb Z).$
Here, $\mathcal{F}= \widehat{\cdot}$ and $\mathcal{F}^{-1}$ denote the Fourier and inverse Fourier transforms, respectively. The \textbf{weighted modulation spaces}  $M^{p,q}_s \ (1 \leq p,q \leq \infty, s \in \R)$ are defined as follows:}
\begin{equation*}
    M^{p,q}_s= M^{p,q}_s(\R)= \left\{ f \in \mathcal{S}'(\R): \left|\left|f\right|\right|_{M^{p,q}_s}=  \left\| \|\square_kf\|_{L^p_x} (1+|k|)^{s} \right\|_{\ell^q_k}< \infty  \right\} . 
\end{equation*}
For $s=0,$ we write $M^{p,q}_0= M^{p,q}.$ For $p=q=2,$ the modulation spaces coincide with \textbf{Sobolev spaces}, i.e. $M^{2,2}_s= H^s \ (s \in \R).$ We also note, unlike $L^p(\mathbb R),$ modulation spaces have the following inclusion relations:
    \begin{equation}\label{srp1}  M_{s_1}^{p_{1},q_{1}}\hookrightarrow M_{s_2}^{p_{2},q_{2}} \quad \text{whenever} \quad
    \begin{cases}
        p_{1}\leq p_{2}, q_{1}\leq q_{2},  s_2 \leq s_1,\\
        p_{1}= p_{2}, q_{2}<q_{1},  s_1-s_2 > \frac{1}{q_2}-\frac{1}{q_1}, 
          \end{cases}
       \end{equation}
for $p_{i},q_{i} \in [1,\infty], (i=1,2)$ and $s_1, s_2 \in \R.$ Refer to \cite{KassoBook,wang2011harmonic} for more details.
Another most interesting property, among others, is the fact that the Schrödinger propagator is uniformly bounded on $M^{p,q}$ spaces for all $p,q$, cf. \cite{chendcds,wango}. Particularly,
\begin{eqnarray}\label{Mpp1} \|e^{i t \partial_x^2}f \|_{M^{p,p'}}\leq  (1+t^{2})^{\frac{1}{2}\left| \frac{1}{p}-\frac{1}{2} \right|}  \|f\|_{M^{p,p'}}.
\end{eqnarray} 
Indeed, this fact enables us to obtain well-posedness results without performing a twisted transformation, whereas, as mentioned above, it was necessary in $ L^p$-spaces.

\subsubsection{\textbf{Known results and gap}}\label{krg}
\par{We briefly recall the following locally well-posedness (LWP for short) and globally well-posedness (GWP) results for \eqref{NLS} and \eqref{NLSH}}.

\begin{theoremA}\label{TH} Let $1\leq p, q \leq \infty$ and $s\in \R.$ Then cubic \eqref{NLS} $(\alpha=3)$ is:
\begin{table}[htbp]
\centering
\begin{tabular}{p{3cm} p{8cm} p{3cm}}
\toprule
&{LWP in $M^{p,q}_{s}$} & \\
\midrule
$s\geq 0$&$ q=1$& \cite{Wang2006, Kasso2009, Bhimani2016,CorderoJMAA} \\
$s>1/q'$ & $1\leq p, q \leq \infty$& \cite{Wang2006, Kasso2009, Bhimani2016,CorderoJMAA} \\
$s=0$&$ 1/q > |1-2/p|$& \cite{Klaus,guo20171d} \\
$s=0$&$p \geq 4, 1/q \geq 1-2/p$& \cite{Klaus} \\
\bottomrule
& GWP in $M^{p,q}_{s}$ &  \\
\bottomrule
$s=0$ & $ 2<p<7/3, q=p'$ &\cite{LeonidIn,leonidthesis}\\
$s=0 $ & $p=2, 1\leq q< \infty$ &\cite{oh2018global}\\
$s=0$ & $ 1\leq q\leq  p \leq 2$ &\cite{Klaus}\\
$s>1/q'$ & $ 1\leq p \leq 2$ &\cite{Klaus}\\
$s>2-1/q$ & $ 2\leq p < \infty$ (defocusing case) & \cite{R.Schippa} \\
$s=1$ & $ 2\leq p< \infty, q=1$ ( defocusing case)&\cite{Klaus}\\
\bottomrule
\end{tabular}
\end{table}
\end{theoremA}
\begin{Remark}[gap and improvement]\label{gai}\
\begin{enumerate}
    \item  For $p<2,$ by Theorem \ref{TH} (\cite{guo20171d,Klaus}), we have  LWP for cubic \eqref{NLS} in $M^{p,q}$ for  $1<p<2, q<\frac{p}{2-p}$. In particular, this gives LWP in $M^{p,p'}$ for $3/2<p\leq2$. To the best of the authors' knowledge,  LWP  in $M^{p,p'}$ for  $p\in [1,3/2]$ is open so far. Theorem \ref{betterregularityNLS} fills this gap up to 4/3, i.e. we could establish LWP in $M^{p,p'}$ for $4/3<p\leq 2$. Moreover, Corollary \ref{globalprnls} extend this  LWP result  globally in $M^{p,p'}$ for $9/5<p\leq 2$. See Figure \ref{hf}.
    \item  The cubic \eqref{NLS} is shown \cite{guo20171d} to be LWP in $M^{2,q}$ with $2 \leq q <\infty$  by establishing trilinear estimates of the cubic nonlinearity $|u|^{2}u$ in the Bourgain space adapted to $M^{2,q}$. Later, invoking  multi-linear interpolation theory among appropriate  modulation spaces, Klaus \cite{Klaus} successfully established LWP of cubic \eqref{NLS} in $M^{p,q}$ for $1<p<2$ and $q<\frac{p}{2-p}$.  Their analysis is confined to the cubic \eqref{NLS} only, while Theorem \ref{betterregularityNLS}  covers \eqref{NLS} with the wider subcritical regime   $\alpha \in (1,5)$.
\item \label{giaRM} Chaichenets et al. \cite{LeonidIn,leonidthesis} studied \eqref{NLS} in $M^{p,p'}$ for $p>2$. Although the solutions obtained in their case take values in different modulation spaces, and have a \textit{loss of regularity}. It can be interpreted as follows: for data   $u_{0} \in M^{p,p'}, (p>2)$,  
the solution $u(\cdot,t)$  lies in larger modulation spaces $$(L^2+ M^{p_{0},p_{0}'})   \left(\supset M^{p, p'}\right) \quad \text{for some}\quad p_{0}\geq p.$$ 
Building on their results, we consider studying \eqref{NLS} in $M^{p,p'}$-spaces for $p < 2$ and obtain solutions in the same modulation spaces (and thus obtain well-posedness in the sense of  Hadamard). Refer to Corollary  \ref{globalprnls} below.
\item Our main concern in this paper is to extend the known results on $M^{p,q}$ with $1/p+1/q=1$ and to cover a wider range of $p$ while improving the regularity. It is well highlighted in Figure \ref{hf} by red lines.
\end{enumerate}    
\end{Remark}

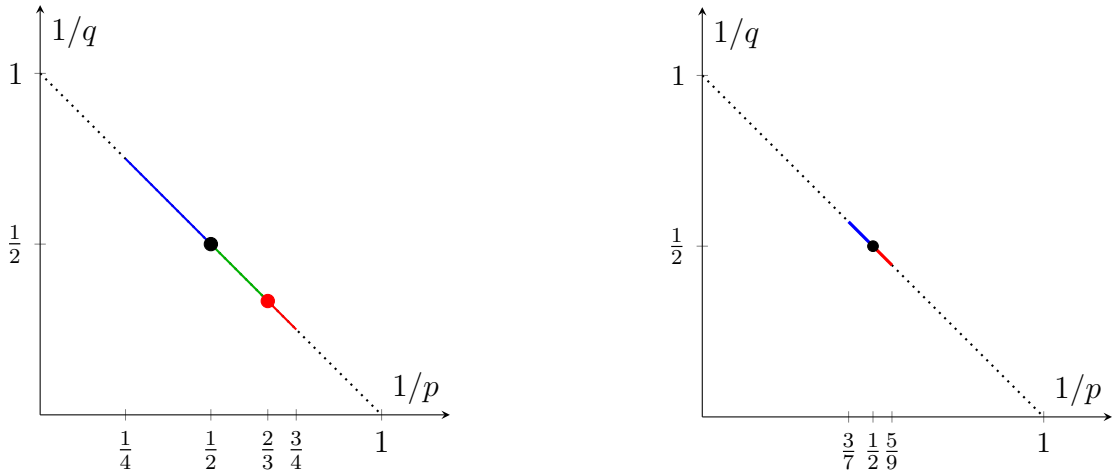
\begin{figure}[htbp]
\centering
\begin{minipage}[b]{0.45\textwidth}
\centering
\begin{tikzpicture}
\begin{axis}[
    axis lines=middle,
    xmin=0, xmax=1.20,
    ymin=0, ymax=1.20,
    xlabel={$1/p$},
    ylabel={$1/q$},
    xtick={0,1/4,1/2,2/3,3/4,1},
    xticklabels={$0$,$\tfrac14$,$\tfrac12$,$\tfrac23$,$\tfrac34$,$1$},
    ytick={0,1/2,1},
    yticklabels={$0$,$\tfrac12$,$1$},
    width=7cm,
    height=7cm,
]
\addplot[
    domain=0:1,
    samples=200,
    dotted,
    thick
] {1 - x};

\addplot[
    thick,
    blue
] coordinates {(1/4,3/4) (1/2,1/2)};
\node[blue] at (axis cs:3/8,5/8) {};
\addplot[
    thick,
    green!70!black
] coordinates {(1/2,1/2) (2/3,1/3)};
\node[green!70!black] at (axis cs:7/12,5/12) {};
\addplot[
    thick,
    red
] coordinates {(2/3,1/3) (3/4,1/4)};
\node[red] at (axis cs:17/24,7/24) {};
\addplot[
    only marks,
    mark=*,
    mark size=2.5pt,
    black
] coordinates {(1/2,1/2)};
\addplot[
    only marks,
    mark=*,
    mark size=2.5pt,
    red
] coordinates {(2/3,1/3)};
\end{axis}
\end{tikzpicture}
\end{minipage}
\hfill
\begin{minipage}[b]{0.45\textwidth}
\centering
\begin{tikzpicture}
\begin{axis}[
    axis lines=middle,
    xmin=0, xmax=1.20,
    ymin=0, ymax=1.20,
    xlabel={$1/p$},
    ylabel={$1/q$},
    xtick={0,3/7,1/2,5/9,1},
    xticklabels={$0$,$\tfrac{3}{7}$,$\tfrac{1}{2}$,$\tfrac{5}{9}$,$1$},
    ytick={0,1/2,1},
    yticklabels={$0$,$\tfrac{1}{2}$,$1$},
    width=7cm,
    height=7cm,
    tick label style={font=\small},
]
\addplot[
    dotted,
    thick,
    domain=0:1,
    samples=200
] {1 - x};
\addplot[
    blue,
    very thick
] coordinates {
    (3/7,4/7)
    (1/2,1/2)
};
\node[blue, above left] at (axis cs:0.43,0.57) {};
\addplot[
    red,
    very thick
] coordinates {
    (1/2,1/2)
    (5/9,4/9)
};
\node[red, below right] at (axis cs:0.54,0.46) {};
\addplot[
    mark=*,
    black
] coordinates {(1/2,1/2)};
\end{axis}
\end{tikzpicture}
\end{minipage}
\caption{Left: local well-posedness of cubic \eqref{NLS} in $M^{p,p'}$. Right: global well-posedness of cubic \eqref{NLS} in $M^{p,p'}$. Blue and green lines denote the results of \cite{LeonidIn,leonidthesis} and \cite{Klaus}, respectively. The red line denotes the results obtained in Theorem \ref{betterregularityNLS} (left) and Corollary \ref{globalprnls} (right), respectively.}
\label{hf}
\end{figure}

\begin{theoremA}[well-posedness of \eqref{NLSH}]\label{ResultsHartree} Let $0<\nu<1$. Then, we have the following well-posedness for the Hartree equation \eqref{NLSH}:
    \begin{enumerate}
   \item \cite{Carles} GWP in $L^2$ and $L^2\cap \widehat{L}^{\infty}$ for  $0<\nu< 1/2$ where $\widehat{L}^{\infty}=\{f \in \mathcal{S'}: \hat{f}\in L^1\}.$ 
   \item \cite{HyakunaIUMJ} (twisted) GWP in $L^{p}$ for $4/3<p \leq 2$ and sufficiently close to 2.
\item \label{HartreeMpq}\cite{bhimani2016cauchy, DGBJDE, manna2017modulation, BhimaniHartree-Fock}  GWP in $M^{p,q}$ with $1 \leq p \leq 2$ and $ 1 \leq q \leq 
\frac{2}{1+\nu}.$ 
\item \label{RM1}\cite{RameshManna} GWP in $M^{p,p'}$ with $2<p$ and sufficiently close to 2.
\item \label{DBF} \cite{DGBJMAA'26} GWP in $M^{p,q} \cap L^2$ with $p,q \in [1,\infty]$ and $0<\nu <1$.
\item \cite{DiracHartree} LWP for Dirac equations; while \cite{BhimaniHartree-Fock} treated \eqref{NLSH} in the presence  of harmonic potential. 
\end{enumerate}
\end{theoremA}
For the ill-posedness of \eqref{NLS} and \eqref{NLSH} with negative regularity in modulation spaces, see \cite{BhimaniNorm, BhimaniJFA2}. For \eqref{NLS} in the presence of a quadratic potential, see \cite{FabioQuo, IvanBook}.
 It would be impossible to mention all the results that have contributed to the success of modulation spaces in the study of dispersive PDEs.  We refer to monographs \cite{KassoBook, wang2011harmonic, IvanBook} for a thorough introduction and the survey article by Ruzansky-Sugimoto-Wang \cite{RuzSurvey}.
\subsection{Statement of the main results}\label{mainresultssection} 
To achieve our main results, we construct a solution to \eqref{NLS} in generalised Strichartz spaces $L^{Q_{p_{0}}(r)}L^{r}$ using ``generalised Strichartz estimates in Fourier-Lebesgue spaces,'' see Subsection \ref{novelty3} below. Its origins go back to the work of Fefferman \cite{CFeffermanAM}; and later, further developed and used successfully by  Gr\"unrock \cite{grunrock2004improved} and  Hyakuna \cite{hyakuna2012existence, HyakunaJFA}. We begin with the following definition.
\begin{Definition}(Generalized Strichartz pair \cite{hyakuna2012existence, CFeffermanAM,grunrock2004improved}). \label{defn2}Let $1\leq p_{0} \leq 2.$
 Consider a pair of exponents $(Q_{p_{0}}(r),r)$ satisfying 
\begin{align*}
&\frac{2}{Q_{p_{0}}(r)} + \frac{1}{r} = \frac{1}{p_{0}} \\
 \hspace{-1cm} \text{either}\quad 
& 0 < \frac{1}{Q_{p_{0}}(r)} < \min \left( \frac{1}{2} - \frac{1}{r},\; \frac{1}{4} \right) \quad  \text{and}\quad
  0 < \frac{1}{r} < \frac{1}{2}\\
& \hspace{-1cm} \text{or}\quad 
(Q_{p_{0}}(r), r)= (4, r) \quad  \text{and}\quad r > 4.
\end{align*}
The set of all such pairs is denoted by 
$\widehat{\mathcal{X}}(p_{0}).$ 
\end{Definition}
\begin{Remark}\label{chipAp}
    Taking  $p_0=2$ in Definition \ref{defn2},  we recover classical admissible pairs  $\mathcal{A}=\widehat{\mathcal{X}}(2)$, see Definition \ref{defn1}. While for $4/3<p_0<2$ and $2<r$\footnote{Note that $r>2$ in all of the three cases considered in \eqref{rvalues}, and \eqref{rnu}.}, we have $Q_{p_{0}}(r)<q(r)=Q_{2}(r)$.  Thus, $\mathcal{A}\subset \widehat{\mathcal{X}}(p_{0})$. 
\end{Remark}
The possible values of $1/p_{0}$ and $1/r$ qualifying for $(Q_{p_{0}}(r), r))\in \widehat{\mathcal{X}}(p_{0})$ are illustrated by Figure \ref{fig2:myplot}:
\begin{figure}[htbp]
\centering
\begin{tikzpicture}[scale=6.7]
\draw[->] (0,0) -- (1,0) node[right] {$\frac{1}{p_{0}}$};
\draw[->] (0,0) -- (0,0.7) node[above] {$\frac{1}{r}$};

\foreach \x/\label in {0/0, 0.5/{\frac{1}{2}}, 0.75/{\frac{3}{4}}}
    \draw (\x,0) -- (\x,-0.02) node[below] {$\label$};
\foreach \y/\label in {0.25/{\frac{1}{4}}, 0.5/{\frac{1}{2}}}
    \draw (0,\y) -- (-0.02,\y) node[left] {$\label$};
\fill[cyan!40]
(0.5,0.25) --
(0.75,0.25) --
(0.5,0.5) -- cycle;

\node at (0.583,0.333) {$R_1$};
\fill[green!40]
(0.5,0) --
(0.5,0.25) --
(0.75,0.25) -- cycle;
\node at (0.583,0.167) {$R_2$};
\draw[dotted, thick] (0.5,0.5) -- (0.75,0.25);
\draw (0.5,0.25) -- (0.75,0.25);
\draw (0.5,0) -- (0.5,0.5);
\draw[red, thick] (0.5,0) -- (0.75,0.25);
\draw[fill=white] (0.5,0.5) circle (0.012);
\draw[fill=white] (0.75,0.25) circle (0.012);
\fill (0.5,0) circle (0.012);
\fill (0.5,0.25) circle (0.012);
\end{tikzpicture}
 \caption{The cyan region $R_{1}$ is defined as the area  bounded by $1/4\leq 1/r< 1/2,\; 1/2 \leq 1/p_{0} < 3/4 \;  \text{and} \; 1/{p_{0}}+1/r<1.$ The green region $R_{2}$ represents the area bounded by $0 < 1/r \leq 1/4,\; 1/2 \leq 1/{p_{0}}< 3/4 \; \text{and}\; 1/{p_{0}}-1/r< 1/2$, while the red line denotes $0 \leq 1/r < 1/4,\;1/{p_{0}}-1/r= 1/2$. }
 \label{fig2:myplot}
\end{figure}
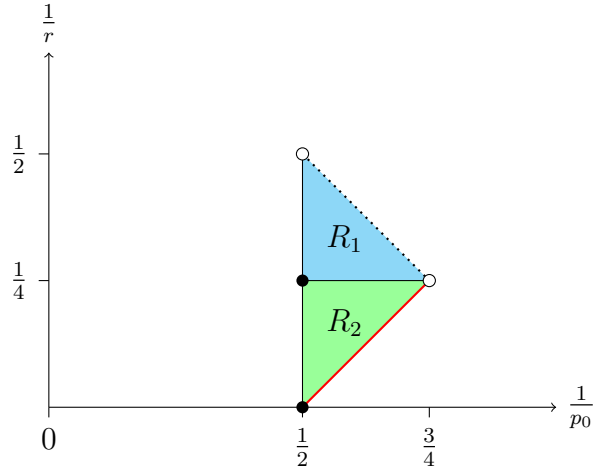

\subsubsection{\textbf{Nonlinear Schr\"odinger equations}}
\par{ For \eqref{NLS}, we divide our study into three cases for different ranges of $\alpha$, which, in turn, decide the value of $r-$the generalised Strichartz space exponents $L^{Q_p(r)}L^r$.
Specifically, for $\alpha \in (1,5),$ we define $r=r(\alpha)$ as}
 \begin{equation}\label{rvalues}r:= 
\begin{cases}
 \alpha+1 &\text{if} \quad \alpha \in  (1,3] \\ 
 4 \quad  & \text{if} \quad \alpha\in (3,11/3]\\ 
 \frac{12}{11}\alpha \quad  & \text{if} \quad \alpha \in (11/3,5).\\ 
\end{cases}
\end{equation}
 The rationale for selecting $r$  will be given in  Remark \ref{whyrso}.
We are now ready to state the first main result for \eqref{NLS}.
 \begin{thm}\label{betterregularityNLS}
  Let $1<\alpha<5$ and $r$ be  defined as \eqref{rvalues}. Assume that  $u_{0}\in M^{p,p'}(\mathbb R)$ with
 $p$ such that
 \begin{gather}\label{ppor}2\geq p>
    \begin{cases}
     \frac{\alpha+1}{\alpha} \quad &\text{if}\quad \alpha \in (1,3]\\
     \frac{4}{6-\alpha}\quad &\text{if}\quad \alpha \in (3,11/3]\\
     \frac{12}{7}\quad &\text{if}\quad \alpha \in (11/3, 31/7]\\
    \frac{\alpha-1}{2}\quad &\text{if}\quad \alpha \in (31/7,5).
    \end{cases}
    \end{gather}
    Then, for $ (Q_{p}(r),r)\in \widehat{\mathcal{X}}(p),$ there exists $T^*=T^*(\|u_{0}\|_{M^{p,p'}},\alpha)>0$ and a unique  solution $u$ of \eqref{NLS} such that 
     \begin{equation*}
    u \in  C([0,T^{*}),M^{p,p'}) \cap L^{Q_{p}(r)}([0,T^{*}),L^{r}) .
     \end{equation*} 
     Moreover, 
     \begin{enumerate}
     \item  (Blow-up alternative)  Either $T^{*} = +\infty$ or $\|u(\cdot,t)\|_{M^{p,p'}} 
     \to \infty$ as $t \to T^{*}$ for $T^*<\infty$.
     \item  (Lipschitz continuity)
The mapping $u_0\mapsto u(t)$ is locally Lipschitz from $M^{p,p'}$ to $C([0,T^{'}],M^{p,p'}) \cap L^{Q_{p}(r)}([0,T^{'}],L^{r})$ for $T'<T^{*}.$
     \end{enumerate}
 \end{thm}
\begin{Remark}
The justification for the lower bound of $p$ in Theorem \ref{betterregularityNLS} is provided in Remark \ref{why1}. For the comments on proof, we refer to  Remark \ref{ps1} below. 
\end{Remark}
\par{With the local solution provided by Theorem \ref{betterregularityNLS}, our next goal is to extend the solution globally with the desired persistence of regularity (see Corollary \ref{globalprnls}).
In light of the blow-up alternative in Theorem \ref{betterregularityNLS}, establishing the global existence reduces to proving that 
\begin{equation}\label{finiteblowup}
    \sup_{0 \leq t < T^*} \|u(t)\|_{M^{p,p'}} < \infty.
\end{equation}
  The main challenge in establishing global well-posedness in modulation spaces is the lack of any useful conservation law in the $M^{p,q} $- norm, cf. \eqref{mass}. To handle this,  we shall first establish Theorem \ref{mr}, although this may be of independent interest. This gives us a global solution (Theorem \ref{mr} \eqref{gwp}) by employing the data-decomposition method for a certain range of $p$. Vargas and Vega introduced this method \cite{vargas2001global} to handle the infinite-mass case. It was initially inspired by Bourgain high-low decomposition  \cite{Bourgain1999} \footnote{This circle of ideas is known as Bourgain's high-low frequency decomposition method in dispersive PDEs, see \cite{Bourgain1999}, \cite[Section 3.9]{TaoBook} \cite[Section 3.2]{KenigonBourgain}.} to extend the global existence result for \eqref{NLS} from $H^1$ to $H^{s}$ for $s>3/5$. Since then, many authors have successfully adapted this method to study \eqref{NLS} and \eqref{NLSH} in non-$L^2$-based spaces; see, for example, \cite{hyakuna2012existence, Grunrock, grunrock2004improved, HyakunaJFA, LeonidIn, HyakunaIUMJ}. }

Returning to Theorem \ref{mr} \eqref{gwp},  we point out that there is a loss of regularity in the global solution. See Remark \ref{remarksnls} \eqref{2solnspace}. Now, in order to obtain the desired persistence of regularity, we extend the local solution in Theorem \ref{betterregularityNLS} globally
with the help of Theorem \ref{mr} \eqref{gwp}. To this end, we begin by noticing that local solutions obtained in  Theorems \ref{betterregularityNLS} and  \ref{mr} are equal almost everywhere. See Proposition \ref{uniqueNLS}. Next, we could control the growth of $M^{p,p'}-$norm (i.e. \eqref{finiteblowup}), by employing  weighted Strichartz estimates in $L^{p} $ spaces \eqref{weightest} and  \eqref{Mpp1}. Indeed, under certain restrictions on $\alpha$ and $p,$ a priori we can bound it by the solution space of Theorem \ref{mr} \eqref{gwp}. See \eqref{essential} and \eqref{connetions}.  This eventually gives us  \eqref{finiteblowup} (and so does Corollary \ref{globalprnls}).

\par{We are now ready to present Theorem \ref{mr} and Corollary \ref{globalprnls} sequentially.}
\begin{thm}\label{mr} Let $1<\alpha<5$ and $r$ be as in \eqref{rvalues}. Assume that  $u_{0}\in M^{p,p'}(\mathbb R)$ with
 $p\in (p_{0},2]$ for some $p_{0}$ satisfying 
\begin{gather}\label{pminloc}2>p_{0}>
    \begin{cases}
     \frac{\alpha+1}{\alpha} \quad&\text{if}\quad \alpha \in (1,3)\\
     \frac{4}{3}\quad&\text{if}\quad \alpha \in [3,10/3]\\
    \frac{2\alpha}{5}\quad&\text{if}\quad \alpha \in (10/3,5).
    \end{cases}
    \end{gather}
 \begin{enumerate}
     \item \label{lwp}
Then, for $ (Q_{p_{0}}(r),r)\in \widehat{\mathcal{X}}(p_{0}),$
 there exists $N_{0} > 1$ such that for any $N > N_{0},$ $\eqref{NLS}$  has a unique local solution 
    \begin{align*}
    u \in \left(L^{\infty}_{T_{N}} L^2 \cap L^{Q_{p_{0}}(r)}_{T_{N}} L^{r}\right) +\left(L^{\infty}_{T_{N}} M^{p_{0},p_{0}'} \cap L^{Q_{p_{0}}(r)}_{T_{N}} L^{r}\right) 
     \end{align*} 
     where $T_{N}$ is represented as
     \begin{equation}\label{TNmain}
T_{N}:=CN^{1+\frac{2}{5-\alpha}\left(4-2\alpha-\frac{\alpha-1}{p_{0}}\right)\gamma} \quad \text{and}\quad \gamma=\frac{\frac{1}{p}-\frac{1}{2}}{\frac{1}{p_{0}}-\frac{1}{p}}.
     \end{equation} 
 \item \label{gwp}
Assume further that $p$ satisfies
\begin{gather}\label{pminglobal}2\geq p>p_{\min}=
    \begin{cases}
     \frac{\alpha+1}{\alpha} \quad&\text{if}\quad \alpha \in (1,\frac{3+\sqrt{57}}{6})\\
     \frac{5\alpha+3}{2(\alpha+2)} \quad&\text{if}\quad \alpha \in [\frac{3+\sqrt{57}}{6},3)\\
     \frac{9(\alpha-1)}{2(2\alpha-1)}\quad&\text{if}\quad \alpha \in [3,\frac{10}{3}]\\
    \frac{(\alpha-1)(3\alpha+5)}{2(\alpha^2-2\alpha+5)}\quad&\text{if}\quad \alpha \in (\frac{10}{3},5).
    \end{cases}
    \end{gather}  Then, the local solution of  \eqref{NLS} 
 constructed in \eqref{lwphar} extends globally and 
\begin{equation*}
  u \in  \left(L^{\infty}_{loc}(\R,L^2)\cap L^{Q_{p_{0}}(r)}_{loc}(\R,L^{r})\right) +\left(L^{\infty}_{loc}(\R,M^{p_{0},p_{0}'})\cap L^{Q_{p_{0}}(r)}_{loc}(\R,L^{r}) \right).
\end{equation*}
\end{enumerate}
\end{thm}
\begin{Remark}\label{remarksnls} Theorem \ref{mr} deserves several comments.
\begin{enumerate}
\item \label{whypsolocally}
The conditions imposed on the lower bound of $p_{0}$ in Theorem \ref{mr} arise from Lemmas \ref{lemlwp} and \ref{lemlwp2}, more precisely, $p_{0}>(\alpha-1)/2$ and $p_{0}>2\alpha /5$ respectively, and the fact that $(Q_{p_{0}}(r),r)\in \widehat{\mathcal{X}}(p_{0})$, see Remark \ref{conditionsonp0}.
\item
For an explanation of the reasoning underlying the conditions imposed on $p$ in \eqref{pminglobal} and the comments on proof, refer to Remark \ref{PS2}.
\item \label{2solnspace}
 The solution space in Theorem \ref{betterregularityNLS} lies in
    \begin{equation*}
    u \in  C_{T} M^{p,p'} \cap L^{Q_{p}(r)}_{T}L^{r}\subset C_{T}M^{p,p'}
     \end{equation*} 
     for some $T>0$ and $u_{0} \in M^{p,p'}$ with $p$ satisfying \eqref{ppor}.
    While, the solution $u$ in Theorem \ref{mr} having initial data in $M^{p,p'}$ with $p$ satisfying \eqref{pminglobal}, lies in
      \begin{align*}
   \left(L^{\infty}_{loc}(\R, L^2)  \cap L^{Q_{p_{0}}(r)}_{loc}(\R, L^{r})\right) +\left(L^{\infty}_{loc} (\R, M^{p_{0},p_{0}'}) \cap L^{Q_{p_{0}}(r)}_{loc} (\R, L^{r})\right) .
     \end{align*}
     Thus, $$u \in L^{\infty}_{loc}(\R, L^2)+ L^{\infty}_{loc}(\R,M^{p_{0},p_{0}'})
     (\supset L^{\infty}_{loc}(\R, M^{p,p'})). $$
      The loss of regularity in Theorem \ref{mr} is due to  data-decomposition method where the initial data $M^{p,p'}$ 
is decomposed as the sum of an arbitrarily large $L^2$ function and an arbitrarily small
$M^{p_{0},p_{0}'}$ function for some $p_{0}<2$, see Lemma \ref{ipt}.
\item \label{compared} Comparing the lower bound on $p$; Theorem \ref{betterregularityNLS}: \eqref{ppor} and Theorem \ref{mr}: \eqref{pminglobal}, it is worth noting that we have a global solution in Theorem \ref{mr} at the cost of a narrower range of $p$ along with loss of regularity.
\end{enumerate}
\end{Remark}

Next, we extend the local solution obtained in Theorem \ref{betterregularityNLS} globally, having persistence of regularity under a certain restrictions on $p$ and $\alpha$. This addresses the limitation highlighted in Remark \ref{remarksnls} \eqref{2solnspace} and \eqref{compared}. Specifically, we have the following result. 

\begin{cor}\label{globalprnls}
  Let $1<\alpha \leq10/3$  and $r$ be as defined in \eqref{rvalues}. Assume that  $u_{0}\in M^{p,p'}(\mathbb R)$ with
 $p$ satisfying \eqref{pminglobal}.
Then, for $ (Q_{p}(r),r)\in \widehat{\mathcal{X}}(p),$ there exists a unique global solution to \eqref{NLS} 
  \begin{equation*}
    u \in C(\R,M^{p,p'}(\mathbb R)) \cap L^{Q_{p}(r)}_{loc}(\R,L^{r}(\R)). 
     \end{equation*} 
\end{cor}
The reason for the restrictions on $\alpha$ is explained in Remark \ref{why10/3}. Finally, we would like to point out that  Theorems \ref{betterregularityNLS} and \ref{mr} and Corollary \ref{globalprnls} nicely complement several well-posedness results and also shed new light on ill-posedness results. See Subsection \ref{mb} and \ref{krg}  and Remark \ref{gai}.

\subsubsection{\textbf{Hartree equations}}
 \begin{thm}\label{betterregularityH}
 For $0<\nu<1$, let
\begin{equation}\label{rnu}
    r_{\nu} =\frac{4}{2-\nu}.
\end{equation}
 Assume that  $u_{0}\in M^{s,s'}(\mathbb R)$ with
 $s$ satisfying 
 \begin{equation}\label{sbetterregularity}
     \frac{4}{2+\nu}<s\leq 2.
 \end{equation}    
 Then,  for $ (Q_{s}(r_{\nu}),r_{\nu})\in \widehat{\mathcal{X}}(s)$, there exists a $T_{*}=T_{*}(\|u_{0}\|_{M^{s,s'}},\nu)>0$ and a unique maximal local solution to Hartree equation \eqref{NLSH} which lies in 
 \begin{equation*}
    u \in C([0,T_{*}),M^{s,s'}) \cap L^{Q_{s}(r_{\nu})}([0,T_{*}),L^{r_{\nu}}) .
     \end{equation*} 
      Moreover, 
     \begin{enumerate}
     \item  (Blow-up alternative)  Either $T_{*} = +\infty$ or $\|u(\cdot,t)\|_{M^{s,s'}} 
     \to \infty$ as $t \to T_{*}$ for $T_*<\infty$.
     \item  (Lipschitz continuity)
The mapping $u_0\mapsto u(t)$ is locally Lipschitz from $M^{s,s'}$ to $C([0,T^{'}],M^{s,s'}) \cap L^{Q_{p}(r)}([0,T^{'}],L^{r})$ for $T'<T_{*}.$
\end{enumerate}
 \end{thm}

 \begin{thm}\label{mrh}
    Let $0<\nu<1$ and $r_{\nu}$ defined in \eqref{rnu}. Assume $u_{0}\in M^{s,s'}(\mathbb R)$ where $s\in (s_{0},2]$ with some $s_{0}$ satisfying
    \begin{equation}\label{s0lwp}
        \frac{4}{2+\nu} <s_{0} < 2.
        \end{equation}
    \begin{enumerate}
        \item\label{lwphar}
        Then, for $ (Q_{s_{0}}(r_{\nu}),r_{\nu})\in \widehat{\mathcal{X}}(s_{0}),$
 there exists  $T_{N}>0$ and $N_{0} > 1$ such that for any $N > N_{0},$ \eqref{NLSH}  has a unique local solution 
    \begin{equation*}
   u\in  \left(L^{\infty}_{T_{N}} L^2 \cap L^{Q_{s_{0}}(r_{\nu})}_{T_{N}} L^{r_{\nu}}\right) +
   \left(L^{\infty}_{T_{N}} M^{s_{0},s_{0}'}\cap L^{Q_{s_{0}}(r_{\nu})}_{T_{N}} L^{r_{\nu}}\right),
     \end{equation*} 
     where $T_{N}$ is of the form 
\begin{equation}\label{TNhar}
    T_{N}:=CN^{1-\frac{2}{2-\nu}\left(\frac{1+\nu}{2}+\frac{1}{s_{0}} \right)\tilde{\alpha} } \quad \text{and}\quad \tilde{\alpha}=\frac{\frac{1}{s}-\frac{1}{2}}{\frac{1}{s_{0}}-\frac{1}{s}}.
  \end{equation}
     \item\label{gwphar} Furthermore, for $s$ satisfying
  \begin{equation}\label{s0gwphar}
     \frac{2\nu+16}{8-\nu^{2}+3\nu}<s\leq 2
  \end{equation}
  the local solution of \eqref{NLSH} constructed in part \eqref{lwp} extends globally and 
\begin{equation*}
   u \in \left(L^{\infty}_{loc}(\R,L^2)\cap L^{Q_{s_{0}}(r_{\nu})}_{loc}(\R,L^{r_{\nu}})\right) +\left(L^{\infty}_{loc}(\R,M^{s_{0},s_{0}'})\cap L^{Q_{s_{0}}(r_{\nu})}_{loc}(\R,L^{r_{\nu}})\right). 
\end{equation*}
 \end{enumerate}
\end{thm}
\begin{cor}\label{globalprh}
     Let $0<\nu<1$ and $u_{0}\in M^{s,s'}(\mathbb R)$ where $s$  satisfies \eqref{s0gwphar}. Recall $r_{\nu}$ given in \eqref{rnu}. Then, for $ (Q_{s}(r_{\nu}),r_{\nu})\in \widehat{\mathcal{X}}(s)$, there exists a global solution to \eqref{NLSH} 
 \begin{equation*}
    u \in C(\R,M^{s,s'}(\mathbb R)) \cap L^{Q_{s}(r_{\nu})}_{loc}(\R,L^{r_{\nu}}(\mathbb R)) .
     \end{equation*} 
\end{cor}
Up to now, the well-posedness theory of Hartree equation \eqref{NLSH} in  $M^{s_{1},s_{2}}$ has been studied only for $1 \leq s_{1} \leq 2$ and $1\leq s_{2} <2$ (Theorem \ref{ResultsHartree} \eqref{HartreeMpq}) and in $M^{s,s'}$ for $2<s$ (Theorem \ref{ResultsHartree} \eqref{RM1}). We note that the later result  have the loss of regularity (cf. Remark \ref{gai} \eqref{giaRM}). Theorem \ref{ResultsHartree} \eqref{DBF} gives GWP for all $s_1,s_2$ but with imposing square integrability condition.  
   While Theorems \ref{betterregularityH} and \ref{mrh}  and  Corollary \ref{globalprh} are the first results that  consider data in $M^{s,s'}$ with  $4/3 < s \leq 2$; and the later result gives the persistence of regularity.
\begin{Remark}\label{remarkhartee} We discuss several imposed hypotheses on the above results.
\begin{enumerate}
    \item   For $(Q_{s_{0}}(r_{\nu}),r_{\nu})$ to be in $ \widehat{\mathcal{X}}(s_{0})$, we impose an extra condition on lower bound of $s_{0}$ given as in \eqref{s0lwp}.
\item \label{remarkhartee2}   Assuming further that
  \begin{equation}\label{assumptionhar}
      \tilde{\alpha}<\left(\frac{2}{2-\nu} \left\{\frac{1+\nu}{2}+\frac{1}{s_{0}} \right\}\right)^{-1},
  \end{equation}
   so that the exponent of $N$ in \eqref{TNhar} is positive, a global solution exists for arbitrarily large values of $N$. While substituting the lower bound of $s_{0}$ (see \eqref{s0lwp}) in \eqref{assumptionhar}, we have an upper bound on $\tilde{\alpha}$ which ultimately gives the lower bound of $s$ (see \eqref{s0gwphar}), ensuring a global solution.
    \item Corollary \ref{globalprh}, when compared with Theorem \ref{mrh} \eqref{gwphar}, recovers the same solution space as initial data $M^{s,s'}$ and extend the local solution of Theorem \ref{betterregularityH} globally, for $s$ satisfying 
   \eqref{s0gwphar}.
\end{enumerate}
\end{Remark}
\subsection{\textbf{Comments on the  proof and novelties}}\label{novelty3}
Chaichenets et al. \cite{LeonidIn,leonidthesis} established well-posedness theory for \eqref{NLS} having initial data in  $M^{p,p'}$ for $ p>2$, making use of the auxiliary space $L^{q(p)}_{T} L^{p}$ so that 
\begin{equation}\label{linearpart>2}
e^{it\partial_x^2}u_{0} \in  L^{\infty}_{T} M^{p,p'} \cap L^{q(p)}_{T} L^{p}, \quad (q(p),p)\in \mathcal{A}.
\end{equation}
The above statement follows due to \eqref{Mpp1} and the embedding $M^{p,p'} \hookrightarrow L^{p}$ since $p>2$.
However, when $p<2$, the auxiliary space used for $p>2$ is no longer applicable, and a different space is required.
To overcome this challenge, we introduce an alternative auxiliary space, generalised Strichartz space, denoted as $L^{Q_{p}(r)}_{T}L^{r}$, so that 
\begin{equation}\label{linearpart}
    e^{it\partial_x^2}u_{0} \in  L^{\infty}_{T} M^{p,p'} \cap L^{Q_{p}(r)}_{T} L^{r}, \quad (Q_{p}(r),r)\in \widehat{\mathcal{X}}(p).
\end{equation}        
Moreover, establishing \eqref{linearpart} requires \eqref{Mpp1} and generalised Strichartz estimates in Fourier-Lebesgue spaces, given as 
(see e.g. \cite{hyakuna2012existence, CFeffermanAM, HyakunaJFA,grunrock2004improved})
\begin{eqnarray}\label{gse3}
    \|e^{i t \partial_x^2}f\|_{L^{Q_{p}(r)}_{T}L^{r}} \lesssim \|f\|_{\widehat{L}^{p}}, \quad  (Q_{p}(r),r) \in \widehat{\mathcal{X}}(p)
\end{eqnarray}
for $4/3< p \leq 2,\; \widehat{\mathcal{X}}(p) \neq \emptyset $ and $T>0$. Here \textbf{Fourier-Lebesgue spaces} are given by 
\begin{equation*}
    \widehat{L^{p}}:=\{f \in \mathcal{S'}:\|f\|_{\widehat{L^{p}}} := \|\widehat{f}\|_{L^{p'}}< \infty \}.
\end{equation*}
The $ \widehat{L^{p}}$ spaces are connected with $M^{p,p'}$ via the following embedding relation (cf. \cite[Corollary 1.1.]{kobayashi2011inclusion}):
\begin{equation}\label{embeddings}
     M^{p,p'} \hookrightarrow \widehat{L}^{p}  \quad  \text{for} \quad 1\leq p \leq 2.
\end{equation}
We also point out that the nonlinear estimates to treat the inhomogeneous part of \eqref{NLS}  in Section \ref{NEPH} are more intricate than the previously studied case for $p>2$. For  $p<2$, the key step is to choose $r$ of Strichartz pair $(Q_{p}(r),r)\in \widehat{\mathcal{X}}(p)$  which is very delicate, see \eqref{rvalues} and Remark \ref{whyrso}. This forces to us to perform our analysis for the different ranges of $\alpha$. While, for $p>2,\; r=\alpha+1$ will work for  all $\alpha \in (1,5)$. 
We now conclude by summarising the proof strategy in the following remarks. 
\begin{Remark}\label{ps1} The following  key ideas  enable us to prove Theorem \ref{betterregularityNLS} via fixed point argument.
\begin{enumerate}
    \item[-] To handle the  homogeneous part of \eqref{NLS} in Theorem \ref{betterregularityNLS}, we benefit from \eqref{Mpp1} and generalised Strichartz estimates \eqref{gse3}, obtaining the solution in generalised Strichartz spaces $L_T^{Q_{p(r)}}L^{r}$. See \eqref{linearfinal}.
\item[-]  To bound the Duhamel operator in the modulation spaces norm by $L_T^{Q_{p(r)}}L^{r}$, we employ weighted Strichartz estimates in $L^{p} $ spaces, developed in \cite[Proposition 2.6(i)]{HyakunaJFA}. Specifically, for any $4/3<p\leq 2, (\sigma', \rho')\in \widehat{\mathcal{X}}(p)$ and $F$ such that $s^{\frac{1}{p}-\frac{1}{2}}F(s)\in L^{\sigma}_{T}L^{\rho},$ we have 
   \begin{equation}\label{weightest}
       \sup_{t\in [0,T]}\left|\left| \int_0^t e^{-is\partial_{x}^{2}} F(s)ds\right|\right|_{L^{p}} \lesssim \|s^{\frac{1}{p}-\frac{1}{2}}F(s)\|_{L^{\sigma}_{T}L^{\rho}}.
       \end{equation}
    See \eqref{integral21}.
 \item[-]  To bound the Duhamel operator in the $L_T^{Q_{p(r)}}L^{r}$ norm by itself, we use inhomogeneous Strichartz estimates, see Proposition \ref{gse2}. Refer to Lemma \ref{lemlwp} and \eqref{integral1}.
\end{enumerate}
\end{Remark}

\begin{Remark}\label{PS2}
We briefly outline the main ideas used in the proof of Theorem \ref{mr}.
\begin{enumerate}
    \item[-] We split the initial data $u_0\in M^{p,p'} $ into sum of an arbitrarily large $L^2$ function and an arbitrarily small
$M^{p_{0},p_{0}'}$ function for some $p_{0}<2$.
Specifically, for any $N>1$ and $u_0 \in M^{p,p'}$ with $p\in (p_{0},2)$,  there exists  $\phi_0 \in L^2$ and $ \psi_0 \in M^{p_{0},p_{0}'}$   such that (cf. Lemma \ref{ipt}),
\begin{equation}\label{dp}
    u_0= \phi_0 + \psi_0, \quad  \|\phi_0\|_{L^2} \lesssim N^{\gamma}, \quad \|\psi_0\|_{M^{p_{0},p_{0}'}} \lesssim  \frac{1}{N}
\end{equation}
\begin{equation}\label{betap}
 \hspace{-5cm} \text{where}\quad  \quad \gamma = \frac{\frac{1}{p}-\frac{1}{2}}{\frac{1}{p_{0}}-\frac{1}{p}}.
    \end{equation}
   \item[-] Next, we solve two different \eqref{NLS} having initial data $\phi_0$ and $\psi_0$, and combine the solution solving \eqref{NLS}, as \[u= (v_{0} + w_{0}) + e^{it\partial_x^2} \psi_0.\] 
    Here $v_{0} $ is the $L^{2}$- global solution of \eqref{NLS} with data $\phi_0$. While $e^{it\partial_x^2}\psi_0 \in M^{p_{0},p_{0}'}$ for $t\in \R$ is the linear evolution of $\psi_0$ and $w_{0}$ is the nonlinear interaction term of $\psi_0$, see \eqref{w01}. 
       \item[-] Since $\|\psi_0\|_{M^{p_{0},p_{0}'}}\lesssim N^{-1}$ can be made small, we get $v_{0} + w_{0}$ close to $v_{0}$ in $L^2.$ Thus, we can extend the solution further by redefining initial data decomposition, $v_{0}(T) + w_{0}(T)\in L^{2}$ and $e^{iT\partial_x^2} \psi_0\in  M^{p_{0},p_{0}'}$, satisfying \eqref{dp}. Continuing these iterations with newly defined initial data \eqref{newidk} can lead to the time of existence \eqref{TNmain}. 
       \item[-]  It is evident that when the exponent of $N$ in \eqref{TNmain} is positive, a global solution exists for arbitrarily large values of $N$. Accordingly, we impose the lower bound on $p$ in Theorem \ref{mr} \eqref{gwp} as given in \eqref{pminglobal}, taking into account \eqref{pminloc} and \eqref{TNmain}, to guarantee the global extension of the solution.
       \end{enumerate}
       \end{Remark}

\par{This paper is organised as follows: In Section \ref{np}, we introduce a few notations and preliminaries that will be useful throughout the paper. Section \ref{NEPH} is dedicated to obtaining estimates for the power- and Hartree-type nonlinearities. Section \ref{NLSresults} presents the proof of Theorems \ref{betterregularityNLS}, \ref{mr} and Corollary \ref{globalprnls}. The proof of Theorems \ref{betterregularityH}, \ref{mrh} and Corollary \ref{globalprh} is provided in Section \ref{Hresults}. Finally, a detailed proof of the interpolation lemma \ref{ipt} can be found in Appendix \ref{app}.}

\section{Notations and preliminaries}\label{np}
  \subsection{Notations} 
   The symbol $X \lesssim Y$ means 
 $X \leq CY$ 
for some constant $C>0$. 
While $X \approx Y $ means $C^{-1}X\leq Y \leq CX$ for some constant $C>0.$
 The norm of the space-time Lebesgue spaces $L^{q}([0,T],L^{r}(\R))$ is defined as
$$\|u\|_{L^{q}_{T}L^{r}}:=\|u\|_{L^{q}([0,T],L^{r}(\R))}=\left(\int_{0}^{T} \|u(\cdot,t)\|_{L^{r}(\R)}^{q} dt\right)^{\frac{1}{q}}.$$
We simply write $\|u\|_{L^{q}L^{r}}$ in place of $\|u\|_{L^{q}(\R,L^{r}(\R))}.$ 
 For  $\alpha>1$ and $u,v,w\in \C$, we denote  \begin{equation*}G(u,v,w)=|u+v|^{\alpha-1}(u+v)-|u+w|^{\alpha-1}(u+w),
\end{equation*}
and we have (see e.g. \cite[Lemma 3.9]{leonidthesis})
     \begin{equation}\label{eg} 
         |G(u,v,w)| \lesssim (|u|^{\alpha-1}+|v|^{\alpha-1}+|w|^{\alpha-1})|v-w|.
     \end{equation}
For three space variable functions $f, g$, and $h$, introduce a trilinear form associated with the Hartree-type nonlinearity as
\begin{equation}\label{Hsymbol}
\mathcal{H}_\nu(f, g, h) :=(| \cdot |^{-\nu} * f\bar{g}) \times h \quad \text{and} \quad \mathcal{H}_\nu(f) := \mathcal{H}_\nu(f, f, f).
\end{equation}
\subsection{Strichartz estimates} 
Define the \textbf{ Schr\"odinger propagator $e^{i t \partial_x^2}$} as follows:
\begin{eqnarray}
\label{sg}\label{-1}
e^{i t \partial_x^2}f(x):= \int_{\R} e^{i\pi t|\xi|^{2}} \widehat{f}(\xi) e^{2\pi i  \xi \cdot x}d \xi \quad (f\in \mathcal{S}, t \in \mathbb R).
\end{eqnarray}
\begin{Definition}\label{defn1} A pair $(q(r),r)$ is admissible if  $2 \leq r\leq \infty$ and
$$\frac{2}{q(r)} =   \frac{1}{2} - \frac{1}{r}.$$
The set of all such admissible pairs is denoted by $$\mathcal{A}= \{(q(r),r):(q(r),r) \; \text{is an admissible pair} \}.$$
\end{Definition}
\begin{prop}[Strichartz estimates]\cite{KeelTao1998}\label{fst} Denote
$$DF(x,t):=  e^{i t \partial_x^2}u_{0}(x)  \pm i\int_0^t e^{i(t-s)\partial_x^2}F(x,s) ds.$$
Assume $u_{0} \in L^2$ and $F \in L^{q_2'(r_2')} (I, L^{r_2'}).$  Then for any time interval $I\ni0$ and admissible pairs $(q_j(r_j),r_j)$, $j=1,2,$ 
satisfying 
$$ \|D(F)\|_{L^{q_1(r_1)}(I,L^{r_1})}  \lesssim  \|u_{0} \|_{L^2}+    \|F\|_{L^{q'_2(r'_2)}(I,L^{r'_2})}$$ where $q_j'(r_j')$ and $ r_j'$ are H\"older conjugates of $q_j(r_j)$ and $r_j$
respectively.
\end{prop}

\begin{prop}[Generalized inhomogeneous Strichartz estimates]\cite[Theorem 2.1]{TKatogse}\label{gse2}
Let $2<\kappa<\infty, 1<\rho<2, 0<\frac{1}{\beta}<\frac{1}{2}-\frac{1}{\kappa}$ and $\frac{3}{2}-\frac{1}{\rho}<\frac{1}{\sigma}<1$. Further assume that
$$2+\frac{2}{\beta}+\frac{1}{\kappa}=\frac{2}{\sigma}+\frac{1}{\rho}.$$ Then, for any $T>0,$ the  following estimate holds:
$$\left|\left| \int_{0}^{t} e^{i(t-s)\partial_x^2}F(s) ds \right|\right|_{L^{\beta}_{T}(L^{\kappa})} \lesssim \|F\|_{L^{\sigma}_{T}(L^{\rho})}.$$
\end{prop}
\begin{Remark}\label{difference}
  When the pairs $(\beta,\kappa)$ and $ (\sigma',\rho')$ are admissible, the above estimates hold, as seen in Proposition \ref{fst}. However, we need the above estimates to hold even for pairs that are not necessarily admissible but satisfy the conditions in Proposition \ref{gse2}.  
\end{Remark}

\section{Nonlinear Estimates}\label{NEPH}
\subsection{Estimates for power nonlinearity}
In this subsection, we establish several lemmas that will play a crucial role in establishing the well-posedness of \eqref{NLS} in the following subsections. Taking Definitions \ref{defn2} and \ref{defn1} into account, we denote the generalised Strichartz space by
\begin{equation}\label{Y(T)}
    Y(T):=L^{Q_{p_{0}}(r)}_{T}L^{r}, \quad (Q_{p_{0}}(r),r) \in \widehat{\mathcal{X}}(p_{0})
\end{equation}
and the Strichartz space by
\begin{equation}\label{Z(T)}
    Z(T):=L^{q(r)}_{T}L^{r}, \quad (q(r),r)\in \mathcal{A} .
\end{equation}
The key challenge is to choose generalised Strichartz space exponents $(Q_{p_{0}}(r),r)$ for a given subcritical power nonlinearity, so that it is compatible with the technicality involved in the proof. We begin with the following useful remark. 
\begin{Remark}\label{whyrso} Recall the values of $r$ as given in \eqref{rvalues}.
\begin{itemize}
    \item[-] In all three cases, the value of  $r$ is chosen to satisfy $(Q_{p_{0}}(r),r)\in \widehat{\mathcal{X}}(p_{0})$, (see Definition \ref{defn2}) and  compatible with the necessary condition $p_{0}>2\alpha/5$ from Lemma \ref{lemlwp2} below. 
    \item[-] Note that when $1<\alpha < 3$ and $(\alpha+1)/\alpha<p_{0},$ we choose $  r=\alpha+1$ so that $(Q_{p_{0}}(\alpha+1),\alpha+1)\in \widehat{\mathcal{X}}(p_{0})$. In this case, the values of $1/p_{0}$ and $1/r$ fall within the region $R_1$ (cyan) of Figure \ref{fig2:myplot}.
     \item[-]For $\alpha=1,$ we cannot take $r=\alpha+1$ as in this case $r=2$ contrary to the assumption $2<r.$ See Definition \ref{defn2}.
    \item[-] Also, $(Q_{p_{0}}(\alpha+1),\alpha+1) \not \in \widehat{\mathcal{X}}(p_{0})$ for any $\alpha \in (3,5)$. Suppose, to the contrary, we choose $r=\alpha+1$ for $\alpha\in (3,5)$ as well. In this case, $r\in (4,6)$, taking Definition \ref{defn2} into account, $Q_{p_{0}}(\alpha+1) \geq 4$.  We can choose $Q_{p_{0}}(\alpha+1)=4$  satisfying 
    \begin{equation}\label{samer}
        \frac{2}{Q_{p_{0}}(\alpha+1)}+\frac{1}{\alpha+1}= \frac{1}{2}+\frac{1}{\alpha+1}=\frac{1}{p_{0}}.
   \end{equation}    
   Considering Lemma \ref{lemlwp2} below, $p_{0}>2\alpha/5$. But, \eqref{samer} doesnt qualify for $p_{0}>2\alpha/5$ when $\alpha \in (3,5)$. Hence, our claim holds.
     \item[-] We need to find another value of $r$ for this range of $\alpha.$ Thus, for $3\leq \alpha \leq 11/3$ and $4/3 <p_{0}$, we choose $r=4$. In this case, the range of  $1/p_{0}$ and $1/r$ lies along the intersecting line between $R_{1}$ (cyan) region and $R_{2}$ (green) region, as depicted in Figure \ref{fig2:myplot}, i.e. $1/r=1/4.$
    \item[-] While for $11/3 <\alpha <5$ and $12\alpha /(6\alpha+11) \leq p_{0}$, we choose $r=12\alpha /11$. The range of $1/p_{0}$ and $1/r$ in this case can be presented by the red line or region $R_{2}$ (green) of Figure \ref{fig2:myplot}.
    \end{itemize}
\end{Remark}    
\begin{Remark}\label{conditionsonp0}
    Concluding Remark \ref{whyrso}, $(Q_{p_{0}}(r),r)\in \widehat{\mathcal{X}}(p_{0})$, provided
     \begin{align}\label{generalp} \left.
     \begin{cases}
     \frac{\alpha+1}{\alpha} &\quad \text{if}\quad 
    \alpha \in (1,3)\\
    \frac{4}{3} &\quad \text{if}\quad 
    \alpha \in [3,11/3)\\
        \frac{12\alpha}{6\alpha+11} &\quad \text{if}\quad 
    \alpha \in [11/3,5)
       \end{cases} \right\}<p_{0}<2.
       \end{align}
\end{Remark}
\begin{lem}\label{lemlwp}
Let $1<\alpha<5$ and $r$ be as defined in \eqref{rvalues}. Then, for any $p_{0}$ satisfying 
\begin{eqnarray}\label{thc}
    \max \left( \frac{\alpha-1}{2}, \frac{\alpha+1}{\alpha}, \frac{4}{3} \right)<p_{0} <2,
\end{eqnarray}
and $(Q_{p_{0}}(r),r)\in \widehat{\mathcal{X}}(p_{0})$, the following estimate holds for every \( T > 0 \):
\begin{align*}
  & \left|\left| \int_0^t e^{i(t-s)\partial_x^2}G(u,v,w)(s) ds \right|\right|_{Y(T)}\\
  & \lesssim \|v-w\|_{Y(T)}\left(T^{1-\frac{\alpha-1}{4}}\|u\|_{Z(T)}^{\alpha-1}
     +T^{1-\frac{\alpha-1}{2p_{0}}}\left(\|v\|_{Y(T)}^{\alpha-1} + \|w\|_{Y(T)}^{\alpha-1}\right)\right)
\end{align*}
for all $u\in Z(T)$ and $v,w \in Y(T).$
\end{lem}
\begin{proof}
By choosing the quadruple $(\beta,\kappa,\sigma,\rho)=(Q_{p_{0}}(r),r,\sigma,r/{\alpha})$ in Proposition \ref{gse2},
    with $p_{0}$ satisfying \eqref{generalp} and 
    \begin{equation}\label{sigma1}
        \frac{1}{\sigma}=1+\frac{1}{Q_{p_{0}}(r)}-\frac{\alpha-1}{2r}=1+\frac{1}{2p_{0}}-\frac{\alpha}{2r},
    \end{equation}
    we have
  \begin{align*}
   \left|\left| \int_0^t e^{i(t-s)\partial_x^2}G(u,v,w)(s) ds \right|\right|_{Y(T)}\lesssim \left|\left| G(u,v,w) \right|\right|_{L^{\sigma}_{T}L^{r/{\alpha}}}. 
\end{align*}
We can easily check that our chosen quadruple satisfies all conditions proposed in Proposition \ref{gse2} for $p_{0}>r/ \alpha$. 
Note that  
\begin{align}
\frac{1}{r/\alpha}=\frac{1}{r}&+\frac{\alpha-1}{r}\label{hc1}\\ \text{and} \quad
\frac{1}{\sigma}=1-\frac{\alpha-1}{4}+\frac{1}{Q_{p_{0}}(r)}+&\frac{\alpha-1}{q(r)}=1-\frac{\alpha-1}{2p_{0}}+\frac{\alpha}{Q_{p_{0}}(r)}\label{hc2}. 
\end{align}
Using  \eqref{eg} and H\"older's inequality twice with H\"older conditions \eqref{hc1} and \eqref{hc2} with respect to space and time variables, respectively, we obtain
\begin{align*}
 &  \left|\left| G(u,v,w) \right|\right|_{L_T^{\sigma}L^{r/{\alpha}}} 
\lesssim \left\| \|v-w \|_{L^r} \| (|u|^{\alpha-1}+|v|^{\alpha-1}+|w|^{\alpha-1}) \|_{L^{r/(\alpha-1)}} \right\|_{L^{\sigma}_T}\\
 &\lesssim \left\{T^{1-\frac{\alpha-1}{4}}\|u\|^{\alpha-1}_{L^{q(r)}_{T}L^r} + T^{1-\frac{\alpha-1}{2p_{0}}}\left(\|v\|^{\alpha-1}_{L^{Q_{p_{0}}(r)}_{T}L^r}+\|w\|^{\alpha-1}_{L^{Q_{p_{0}}(r)}_{T}L^r}\right)\right\}\|v-w \|_{L^{Q_{p_{0}}(r)}_{T}L^r}\\
&= \|v-w\|_{Y(T)}\left(T^{1-\frac{\alpha-1}{4}}\|u\|_{Z(T)}^{\alpha-1}+T^{1-\frac{\alpha-1}{2p_{0}}}\left(\|v\|_{Y(T)}^{\alpha-1} + \|w\|_{Y(T)}^{\alpha-1}\right)\right).
     \end{align*}
In the last inequality, the exponents of $T$ are both positive since $1<\alpha<5$ and $p_{0}>(\alpha-1)/2$ \footnote{Also, $(\alpha-1)/2>12\alpha/(6\alpha+11)$ for $11/3<\alpha<5$.}.
    \end{proof}

Next, we  estimate $L_T^{\infty}L^2-$norm under stronger assumption \eqref{thc2} (cf. \eqref{thc}). 
\begin{lem}\label{lemlwp2}
 Let $1<\alpha<5$ and $r$ be as defined in \eqref{rvalues}. Further assume
\begin{eqnarray}\label{thc2}
    \max\left\{ \frac{2\alpha}{5},\frac{4}{3},\frac{\alpha+1}{\alpha} \right\} < p_{0}.
\end{eqnarray}
Then, for all $ T > 0 $ and $(Q_{p_{0}}(r),r)\in \widehat{\mathcal{X}}(p_{0})$, the following estimate holds:
\begin{align*}
   & \left|\left| \int_0^t e^{i(t-s)\partial_x^2}G(u,v,0)(s) ds \right|\right|_{L^{\infty}_{T}L^{2}}\\
   &\lesssim T^{1-\frac{\alpha-1}{4}-\frac{1}{2}\left(\frac{1}{p_{0}}-\frac{1}{2}\right)}\|u\|_{Z(T)}^{\alpha-1}\|v\|_{Y(T)}
     +T^{1-\frac{\alpha-1}{2p_{0}}-\frac{1}{2}\left(\frac{1}{p_{0}}-\frac{1}{2}\right)}\|v\|_{Y(T)}^{\alpha} .
\end{align*}
with $u\in Z(T)$ and $ v \in Y(T).$
\end{lem}
\begin{proof}
Using Proposition \ref{fst} with admissible pairs, $(\infty,2)$ and $(\sigma',(r/{\alpha})')$, where
\begin{equation*}
        \frac{1}{\sigma}=1+\frac{1}{q(r)}-\frac{\alpha-1}{2r}=1+\frac{1}{4}-\frac{\alpha}{2r},
    \end{equation*}
we have
  \begin{align*}
   \left|\left| \int_0^t e^{i(t-s)\partial_x^2}G(u,v,w)(s) ds \right|\right|_{L^{\infty}_{T} L^{2}}\lesssim \left|\left| G(u,v,w) \right|\right|_{L^{\sigma}_{T}L^{r/{\alpha}}}. 
\end{align*}
The rest of the proof follows the same lines as Lemma \ref{lemlwp}. The main difference lies in the auxiliary result employed: Lemma \ref{lemlwp} relies on Proposition \ref{gse2}, whereas the proof of Lemma \ref{lemlwp2} requires Proposition \ref{fst}. See also Remark \ref{difference} for further clarification.
Note that the exponents of $T$ are positive based on the assumptions $1<\alpha<5$ and $p_{0}>2\alpha /5.$ 
\end{proof}
\subsection{Estimates for Hartree nonlinearity}
 Recall $r_{\nu}$ as defined in \eqref{rnu}. Denote the generalised Strichartz space by
\begin{equation}\label{X1(T)}
    X_{1}(T):=L^{Q_{s_{0}}(r_{\nu})}_{T}L^{r_{\nu}}, \quad (Q_{s_{0}}(r_{\nu}),r_{\nu}) \in \widehat{\mathcal{X}}(s_{0})
\end{equation}
and the Strichartz space by
\begin{equation}\label{X2(T)}
    X_{2}(T):=L^{q(r_{\nu})}_{T}L^{r_{\nu}}, \quad (q(r_{\nu}),r_{\nu})\in \mathcal{A} .
\end{equation}
\begin{lem}\label{lemlwph}
Assume $0<\nu<1$ and $$ \frac{4}{2+\nu} <s_{0}< 2.$$
Then, for every \( T > 0 \) and $(Q_{s_{0}}(r_{\nu}),r_{\nu}) \in \widehat{\mathcal{X}}(s_{0})$, the following estimate holds:
\begin{enumerate}
\item\label{xxx} For $u,v,w \in X_{1}(T),$ we have\begin{align*}
   \hspace{-0.5cm}\left|\left| \int_0^t e^{i(t-s)\partial_{x}^{2}}\mathcal{H}_{\nu}(u,v,w)(s) ds \right|\right|_{X_{1}(T)} \lesssim T^{1-\frac{\nu}{2}-\left(\frac{1}{s_{0}}-\frac{1}{2}\right)}\left(\|u\|_{X_{1}(T)}
     \|v\|_{X_{1}(T)}  \|w\|_{X_{1}(T)}\right).
\end{align*}
\item\label{yxx}For $u\in X_{2}(T)$ and $v,w \in X_{1}(T)$, we have
\begin{align*}
   \hspace{-0.5cm}\left|\left| \int_0^t e^{i(t-s)\partial_{x}^{2}}\mathcal{H}_{\nu}(u,v,w)(s) ds \right|\right|_{X_{1}(T)} \lesssim T^{1-\frac{\nu}{2}-\frac{1}{2}\left(\frac{1}{s_{0}}-\frac{1}{2}\right)}\left(\|u\|_{X_{2}(T)}
     \|v\|_{X_{1}(T)}  \|w\|_{X_{1}(T)}\right).
\end{align*}
\item\label{yyx} For $u,v\in X_{2}(T)$ and $w\in X_{1}(T)$, we have
\begin{align*}
   \hspace{-0.5cm}\left|\left| \int_0^t e^{i(t-s)\partial_{x}^{2}}\mathcal{H}_{\nu}(u,v,w)(s) ds \right|\right|_{X_{1}(T)} \lesssim T^{1-\frac{\nu}{2}}\left(\|u\|_{X_{2}(T)}
     \|v\|_{X_{2}(T)}  \|w\|_{X_{1}(T)}\right).
\end{align*}
\end{enumerate}
\end{lem}
\begin{proof}
By choosing the quadruple $(\beta,\kappa,\sigma,\rho)=(Q_{s_{0}}(r_{\nu}),r_{\nu},\sigma,r_{\nu}')$ in Proposition \ref{gse2}, where $\sigma$ can be determined by the relation $$2+\frac{2}{Q_{s_{0}}(r_{\nu})}+\frac{1}{r_{\nu}}=\frac{2}{\sigma}+\frac{1}{r_{\nu}'},$$ we have 
  \begin{align*}
   \left|\left| \int_0^t e^{i(t-s)\partial_{x}^{2}}\mathcal{H}_{\nu}(u,v,w)(s) ds \right|\right|_{X_{1}(T)}\lesssim \left|\left| \mathcal{H}_{\nu}(u,v,w) \right|\right|_{L^{\sigma}L^{r_{\nu}'}}. 
\end{align*}
We can easily check that our chosen quadruple satisfies all conditions proposed in Proposition \ref{gse2} for $s_{0}>r_{\nu}'$. 
Define $a$ and $b$ as
\begin{align}
    \frac{1}{\sigma}&=\frac{1}{a}+\frac{1}{Q_{s_{0}}}=\frac{3}{Q_{s_{0}}}+1-\frac{\nu}{2}-\left(\frac{1}{s_{0}}-\frac{1}{2}\right)\label{expa}\\
    \frac{1}{r_{\nu}'}&=\frac{1}{r_{\nu}}+\frac{1}{b}=\frac{1}{r_{\nu}}+\frac{2}{r_{\nu}}+\nu-1 \label{expb} 
\end{align}
Further, applying Hardy-Littlewood Sobolev inequality and H\"older's inequality twice, with $a$ and $b$ satisfying \eqref{expa} and \eqref{expb}, respectively, we have
\begin{align*}
    \left|\left| \mathcal{H}_{\nu}(u,v,w) \right|\right|_{L^{\sigma}L^{r_{\nu}'}} 
    & \lesssim \| |\cdot|^{-\nu}\ast (uv) \|_{L^{a}_{T}L^{b}} \|w\|_{X_{1}(T)}   \\
&\lesssim T^{1-\frac{\nu}{2}-\left(\frac{1}{s_{0}}-\frac{1}{2}\right)}\left(\|u\|_{X_{1}(T)}
     \|v\|_{X_{1}(T)}  \|w\|_{X_{1}(T)}\right).
\end{align*}
The exponents of $T$ are positive since we assume
$$\frac{4}{2+\nu}<s_{0} < 2 \quad \text{and} \quad 0<\nu<1.$$

The remaining inequalities follow from H\"older's inequality. Particularly, $\|\cdot\|_{X_{1}(T)} \leq T^{\frac{1}{2}\left(\frac{1}{s_{0}}-\frac{1}{2}\right)}\|\cdot \|_{X_{2}(T)}$. This concludes our claim.
\end{proof}

\section{Proofs of Theorems \ref{betterregularityNLS}, \ref{mr} and Corollary \ref{globalprnls}}\label{NLSresults}

\subsection{Proof of Theorem \ref{betterregularityNLS}}
Based on the Duhamel principle, the solutions to \eqref{NLS} are equivalent to the integral equation 
\begin{equation}\label{DuhamelP}
    u(t)=e^{it\partial_{x}^{2}}u_{0}\pm i  \int_{0}^{t}e^{i(t-s)\partial_{x}^{2}}(|u|^{\alpha-1}u)(s)ds:=\Lambda_{1}(u)(t).
\end{equation}
Taking Remark \ref{conditionsonp0} into account for
     $$\max\left\{\frac{4}{3},\frac{\alpha+1}{\alpha}, \frac{12\alpha}{6\alpha+11} \right\}<p,$$
we have $$(Q_{p}(r),r)\in \widehat{\mathcal{X}}(p).$$
    Define \begin{equation}\label{Y1(T)}
    Y_{1}(T):=L^{Q_{p}(r)}_{T}L^{r}
\end{equation}
and $$X(T):=Y_{1}(T) \cap L^{\infty}_{T}M^{p,p'}$$ 
    equipped with the norm 
$$\|v\|_{X(T)}=\max \left\{\|v\|_{Y_{1}(T)}, \|v\|_{L^{\infty}_{T}M^{p,p'}}\right\} .$$
For $a$ and $T$ be positive real numbers (to be chosen later), denote $$B(a, T):=\{u\in X(T):\|u\|_{X(T)}\leq a\}.$$
We will show that $\Lambda_{1}$ 
    is a contraction map on $B(a,T).$ 
    Firstly, we consider the linear evolution of $u_{0}\in M^{p,p'}$.  Assume w.l.o.g. that $T\leq 1$. Using \eqref{gse3} and \eqref{embeddings}, we have
    \begin{equation}\label{linear1}
         \|e^{i t \partial_x^2}u_{0}\|_{Y_{1}(T)} \lesssim 
        \|u_{0}\|_{M^{p,p'}}.
    \end{equation}
    Combining \eqref{linear1} and \eqref{Mpp1}, we have
    \begin{equation}\label{linearfinal}
        \|e^{i t \partial_x^2}u_{0}\|_{X(T)} \lesssim \|u_{0}\|_{M^{p,p'}}.
         \end{equation}
    This suggests the choice of
    $\label{achoice}
    a=C\|u_{0}\|_{M^{p,p'}}.$
    Using Lemma \ref{lemlwp}, for $u\in B(a,T)$ and $p>(\alpha-1)/2$, we have
    \begin{align}\label{integral1}
    \left|\left| \int_{0}^{t} e^{i(t-\tau)\partial_{x}^{2}}(|u|^{\alpha-1}u)(s)ds \right|\right|_{Y_{1}(T)}&=\left|\left| \int_{0}^{t} e^{i(t-\tau)\partial_{x}^{2}}G(0,u,0)(s)ds \right|\right|_{Y_{1}(T)} \nonumber\\
    &\lesssim   T^{1-\frac{\alpha-1}{2p}}\|u\|_{Y_{1}(T)}^{\alpha}.
    \end{align}
    Next, we bound the Duhamel operator in the modulation spaces norm by $Y_{1}(T)$.  To this end, the idea is to invoke \eqref{weightest} and so we  choose  $b=b(r)$ as follows 
    \begin{equation}\label{HolderB}
        \frac{1}{b}=\frac{3}{2}-\frac{\alpha}{2r}-\frac{1}{2p}=\frac{3p-\alpha-1}{2p}+\frac{\alpha}{Q_{p}(r)}.
      \end{equation}
 In view of this, together with  \eqref{rvalues}, we obtain      
    \begin{gather*}b'=\frac{2rp}{r+p(\alpha-r)}=
        \begin{cases}
       \frac{2p(\alpha+1)}{\alpha+1-p}  &\quad \text{if}\quad \alpha \in (1,3]\\
         \frac{8p}{4+p(\alpha-4)}  & \quad \text{if}\quad \alpha \in (3,11/3]\\
         \frac{24p}{12-p} & \quad \text{if}\quad \alpha \in (11/3,5)
        \end{cases}.
    \end{gather*}
By invoking the   hypothesis on $p$ (see Remark \ref{why1} below), we obtain   $$(b',(r/\alpha)') \in \widehat{\mathcal{X}}(p).$$ 
Now, using \eqref{Mpp1}, \eqref{dsembeddings}, \eqref{weightest} with $(\sigma',\rho')=(b',(r/\alpha)')\in \widehat{\mathcal{X}}(p)$ and H\"older  inequality with \eqref{HolderB}, we obtain 
   \begin{align}
    \left|\left| \int_{0}^{t} e^{i(t-\tau)\partial_{x}^{2}}(|u|^{\alpha-1}u)(\tau)d\tau \right|\right|_{L_{T}^{\infty}M^{p,p'}}
   &= \left|\left| e^{it\partial_{x}^{2}} \left\{\int_{0}^{t} e^{-i\tau\partial_{x}^{2}}(|u|^{\alpha-1}u)(\tau)d\tau \right\}\right|\right|_{L_{T}^{\infty}M^{p,p'}}\nonumber\\
   &\lesssim  \left|\left| \int_{0}^{t} e^{-i\tau\partial_{x}^{2}}(|u|^{\alpha-1}u)(\tau)d\tau \right|\right|_{L_{T}^{\infty}M^{p,p'}} \nonumber\\
    &\lesssim  \left|\left| \int_{0}^{t} e^{-i\tau\partial_{x}^{2}}(|u|^{\alpha-1}u)(\tau)d\tau \right|\right|_{L_{T}^{\infty}L^{p}} \nonumber\\
    &\lesssim \||\tau|^{\frac{1}{p}-\frac{1}{2}}|u(\tau)|^{\alpha}\|_{L_{T}^{b}L^{\frac{r}{\alpha}}}\nonumber\\
     &\lesssim \||\tau|^{\frac{1}{p}-\frac{1}{2}}\|_{L_{T}^{\frac{2p}{3p-\alpha-1}}}\|u\|_{Y_{1}(T)}^{\alpha}\nonumber\\
   &\lesssim \label{integral21} T^{1-\frac{\alpha-1}{2p}}\|u\|_{Y_{1}(T)}^{\alpha}.
    \end{align}    
Thus, combining \eqref{integral1} and \eqref{integral21}, we have
    \begin{equation}\label{integralfinal}
          \left|\left| \int_{0}^{t} e^{i(t-\tau)\partial_x^2}(|u|^{\alpha-1}u)(s)ds \right|\right|_{X(T)} \lesssim T^{1-\frac{\alpha-1}{2p}}\|u\|_{X(T)}^{\alpha}.
    \end{equation}
Taking
    \begin{equation}\label{T*NLS}
        T:=\min \left\{ 1,~C(\alpha) \|u_{0}\|^{-\frac{\alpha-1}{1-\frac{\alpha-1}{2p}}}_{M^{p,p'}} \right \},
    \end{equation}
 we combine \eqref{linearfinal} and \eqref{integralfinal} to conclude $\Lambda_{1}(u)\in B(a, T).$\\
 Using  \eqref{eg} for $G(0,u,v)$ and the previous argument employed, for $u,v \in B(a,T)$, we can similarly establish
 $$\|\Lambda_{1}(u)-\Lambda_{1}(v)\|_{X(T)}\leq \frac{1}{2}\|u-v\|_{X(T)}.$$
Thus, by the contraction mapping theorem, we obtain a unique fixed point of $\Lambda_{1}(u)$, which is a solution to \eqref{NLS}.\\
 For Lipschitz continuity, let $ v_{0}, w_{0} \in V $, where $ V $ is a neighborhood of $ u_{0}$. Denote by $v$ and $w$ the unique maximal solutions of \eqref{NLS} over the interval $[0,T^*)$ with initial values $ v_{0} $ and $ w_{0} $, respectively. Thus, \( v \) and \( w \in X(T')\) for \( T' < T^{*}.  \)
Hence
\begin{eqnarray*}
    \|v-w\|_{X(T')} &=&\|e^{it\partial_x^2}v_{0}-e^{it\partial_x^2}w_{0} + \Lambda_{1}(v) - \Lambda_{1}(w) \|_{X(T')}\\
    & \lesssim & \|v_{0}-w_{0}\|_{M^{p,p'}}+\frac{1}{2}\|v-w\|_{X(T')}.\\
\text{Thus,}\;
\|v-w\|_{X(T')} &\lesssim & \|v_{0}-w_{0}\|_{M^{p,p'}}
\end{eqnarray*}
for any $v_{0}, w_{0} \in V \subset M^{p,p'}.$ This concludes local Lipschitz continuity.
\qed
\begin{Remark}\label{why1} The specific lower bound imposed on $p$ in Theorem \ref{betterregularityNLS} is determined based on the considerations outlined below for different ranges of $\alpha$. 
\begin{enumerate}
\item For $\alpha \in (1,3]$ and $p>(\alpha+1)/\alpha$, $$\left(b',\left(\frac{r}{\alpha}\right)'\right)=\left(\frac{2p(\alpha+1)}{\alpha+1-p},\alpha+1 \right) \in \widehat{\mathcal{X}}(p). $$
The region shaded in cyan ($R_1$) in Figure \ref{fig2:myplot} characterises the range of $1/p_{0}$ and $1/r$.
    \item While for $\alpha \in (3,11/3]$, we have $$\left(b',\left(\frac{r}{\alpha}\right)'\right)
=\left(\frac{8p}{4+p(\alpha-4)},\frac{4}{4-\alpha}\right). $$ Since $\frac{4}{4-\alpha}>4$, considering Definition \ref{defn2}, we require $\frac{8p}{4+p(\alpha-4)} \geq 4$. Thus for $p \geq 4/(6-\alpha),$ we have $ (b',\left(r/\alpha \right)')\in \widehat{\mathcal{X}}(p)$. The values of 
$1/p_{0}$ and $1/r$ are precisely those contained in region $R_{2}$ (green) or on the red line of Figure \ref{fig2:myplot}.
    \item 
 It is important to note that when $\alpha \in (11/3,5)$, $$\left(b',\left(\frac{r}{\alpha}\right)'\right)
 =\left(\frac{24p}{12-p},12\right).$$ Considering Definition \ref{defn2}, we require $\frac{24p}{12-p}\geq 4$. For $p \geq 12/7,$ we have $(b',\left(r/\alpha \right)') \in \widehat{\mathcal{X}}(p)$. The possible values of $1/p_{0}$ and $1/r$ are precisely those contained in the region $R_{2}$ (green) or on the red line of Figure \ref{fig2:myplot}.
 \item \label{whyppor} Taking Remark \ref{conditionsonp0} and the preceding remarks into account, the lower bound of $p$ can be determined as \begin{gather*}2\geq p>
    \begin{cases}
    \max\{ \frac{\alpha+1}{\alpha}, \frac{\alpha-1}{2}\}\quad &\text{if}\quad \alpha \in (1,3]\\
     \max\{ \frac{4}{3},\frac{4}{6-\alpha}, \frac{\alpha-1}{2}\}\quad &\text{if}\quad \alpha \in (3,11/3]\\
    \max\{\frac{12\alpha}{6\alpha+11},  \frac{12}{7},\frac{\alpha-1}{2}\}\quad &\text{if}\quad \alpha \in (11/3, 5).
    \end{cases}
    \end{gather*}
\end{enumerate}
    \end{Remark}

\subsection{Proof of Theorem \ref{mr}}\label{secMR}
To begin with the proof of Theorem \ref{mr} \eqref{lwp}, we start by decomposing the initial data $u_0 $ into $\phi_0 \in L^2$ and $\psi_0 \in M^{p_{0},p_{0}'}$ as done in \eqref{dp} (also see Lemma \ref{ipt}). The solution to \eqref{NLS} having initial data $u_{0}$ can be written as $\phi_{0}+\psi_{0}$. Here, $v_{0}$ and $w_{0}$ denote the solutions corresponding to the initial data 
$\phi_0$ and $\psi_{0}$, respectively. 
Consider \eqref{NLS} with initial data $\phi_0,$ namely
\begin{eqnarray}\label{ivpL2}
\begin{cases}
 i \partial_t v_{0} +(v_{0})_{xx} \pm |v_{0}|^{\alpha-1}v_{0}=0 \\
v_{0}(\cdot,0)=\phi_{0}\in L^2.   
\end{cases}        
    \end{eqnarray}
We recall that Tsutsumi \cite{YTsutsumi} proved that \eqref{ivpL2} has a unique global solution $v_{0}$ in the space
\begin{eqnarray}\label{gwpl2}
\begin{cases}
  v_{0} \in C(\R, L^2) \cap L_{loc}^{q(r)}(\R,L^{r} ),\\
\sup_{(q(r),r)\in \mathcal{A}}\|v_{0}\|_{L_{loc}^{q(r)} L^{r}} \lesssim \|\phi_{0}\|_{L^{2}},
\end{cases}
\end{eqnarray}
where $(q(r),r) \in \mathcal{A}$. Considering Remark \ref{chipAp}, we also have $ v_{0}\in L_{loc}^{Q_{p_{0}}(r)}(\R,L^{r})$ since $v_{0} \in L_{loc}^{q(r)}(\R,L^{r})$.
We now turn to the modified \eqref{NLS} associated with the evolution of $\psi_{0}$: 
\begin{eqnarray}\label{ivpMod}
\begin{cases}
 i \partial_t w_{0} +(w_{0})_{xx} \pm (|w_{0}+v_{0}|^{\alpha-1}(w_{0}+v_{0})-|v_{0}|^{\alpha-1}v_{0})=0 \\
w_{0}(\cdot,0)=\psi_{0}\in M^{p_{0},p_{0}'} .
\end{cases}
\end{eqnarray}
By Duhamel's principle, the above I.V.P. \eqref{ivpMod} is equivalent to the integral equation:
\begin{align}
    w_{0} &= e^{i t \partial_x^2}\psi_{0}\pm i  \int_{0}^{t}  e^{i(t-\tau)\partial_x^2}\left( \big| w_{0}+v_{0}\big|^{\alpha-1} (w_{0}+v_{0}) - |v_{0}|^{\alpha-1} v_{0}\right)(\tau) \, d\tau\nonumber \\
    &= e^{i t \partial_x^2}\psi_{0} \pm i  \int_{0}^{t}  e^{i(t-\tau)\partial_x^2}G(v_{0},w_{0},0) \, d\tau.\label{w01}
\end{align}

 Since $v_0$ is globally defined in appropriate spaces as seen in \eqref{gwpl2}, we only need to examine the time interval of existence for $w_0$ to establish the desired existence.
 In the next proposition,
 we will construct a local solution to \eqref{w01} on a small interval $[0,\delta_{N}]$ defined as 
\begin{equation}\label{TN}
    \delta_{N}\sim N^{-\frac{4(\alpha-1)}{5-\alpha}\gamma}.
\end{equation}

\begin{prop}\label{w0exist} Let $\phi_{0} \in L^2$ and $ \psi_{0}\in M^{p_{0},p_{0}'}$.  Assume that $1<\alpha<5$ and $p_{0}$ are as in Theorem \ref{mr} \eqref{lwp}. 
Then, the integral equation \begin{equation}\label{w012}
    e^{i t \partial_x^2}\psi_{0} \pm i \int_{0}^{t}  e^{i(t-\tau)\partial_x^2}G(v_{0},w_{0},0) \, d\tau
    \end{equation}
    has a unique solution in  $Y(\delta_{N}).$
\end{prop}
\begin{proof}
We use the fixed point theorem in the complete metric space $B_{A}$ defined as
$$B_{A}:=\{w_{0} \in Y(\delta_{N}
   ):\|w_{0}\|_{Y(\delta_{N}
   )}\leq A/N \}$$ for some $A>0.$
  For this, we claim that the operator
    $$\Gamma(w_{0}):=e^{i t \partial_x^2}\psi_{0}\pm i \int_{0}^{t}  e^{i(t-\tau)\partial_x^2}G(v_{0},w_{0},0) \, d\tau$$
    is a well-defined contraction map on $B_{A}$.\\
   Considering \eqref{gse3}, Lemma \ref{lemlwp} and \eqref{gwpl2}, for any $w_{0}\in B_{A},$ we have 
    \begin{align}
    \|\Gamma(w_{0})\|_{Y(\delta_{N})} &\leq \|e^{i t \partial_x^2}\psi_{0}\|_{Y(\delta_{N})}+\left\| \int_{0}^{t} e^{i(t-\tau)\partial_x^2}G(v_{0},w_{0},0) \, d\tau \right\|_{Y(\delta_{N})}\nonumber \\
   \lesssim & \|\psi_{0}\|_{M^{p_{0},p_{0}'}}+\delta_{N}^{1-\frac{\alpha-1}{4}} \|v_{0}\|^{\alpha-1}_{Z(\delta_{N}
   )} \|w_{0}\|_{Y(\delta_{N}
   )}+
     \delta_{N}^{1-\frac{\alpha-1}{2{p_{0}}}}\|w_{0}\|^{\alpha}_{Y(\delta_{N}
   )}\nonumber\\
     \lesssim & \frac{1}{N} + 2N^{\frac{-4(\alpha-1)}{5-\alpha}\left(\frac{5-\alpha}{4}\right)\gamma}N^{(\alpha-1)\gamma}\frac{1}{N} \leq\frac{3}{N}.\nonumber
\end{align}
The last inequality is due to the bounds \eqref{dp} and \eqref{TN}. Thus, $\Gamma(w_{0}) \in B_A.$ Similarly, we can prove that $\Gamma(w)$ is a contraction map. For $w_{01},w_{02} \in B_{A},$ we have 
\begin{align*}
   \hspace{-4cm} \|\Gamma(w_{01})-\Gamma(w_{02})\|_{Y(\delta_{N})} &= \left\| \int_{0}^{t} e^{i(t-\tau)\partial_x^2}G(v_{0},w_{01},w_{02}) \, d\tau \right\|_{Y(\delta_{N})}\nonumber\end{align*}
\begin{align*}
    \lesssim& \|w_{01}-w_{02}\|_{Y(\delta_{N})}\left(\delta_{N}^{1-\frac{\alpha-1}{4}}\|v_{0}\|_{Z(\delta_{N})}^{\alpha-1}
     +\delta_{N}^{1-\frac{\alpha-1}{2p_{0}}}(\|w_{01}\|_{Y(\delta_{N})}^{\alpha-1} + \|w_{02}\|_{Y(\delta_{N})}^{\alpha-1})\right) \nonumber \\
      \lesssim &\|w_{01}-w_{02}\|_{Y(\delta_{N})}.
\end{align*}
Therefore, the contractivity of $\Gamma$ follows similarly \footnote{$C$ is chosen small enough such that all requirements are fulfilled.}.
Using the Banach fixed-point theorem, we get a unique fixed point $w_{0}$ to the integral equation \eqref{w012} on the time-interval $[0,\delta_{N}].$
\end{proof} 
Consequently, we have a local solution $ w_{0} \in L^{Q_{p_{0}}(r)}_{\delta_{N}}L^{r}$ on the time interval $[0,\delta_{N}]$. Thus, based on \eqref{gwpl2}, Proposition \ref{w0exist} and \eqref{Mpp1},  we conclude $$v_0 + (w_0-e^{i t \partial_x^2} \psi_0) \in  L^{Q_{p_{0}}(r)}_{\delta_{N}}L^{r}$$ and $$
    e^{i t \partial_x^2} \psi_0 \in L^{\infty}_{\delta_{N}}M^{p_{0},p_{0}'} .$$

Next, we claim that the local solution  $u(t),\;t\in [0,\delta_{N}]$ obtained in the previous step can be extended to the interval $[\delta_{N},2\delta_{N}]$.
Consider the new initial data decomposition as the sum of the following two functions:    
\[ \phi_{1}=v_{0}(\delta_{N}) \pm i \int_{0}^{\delta_{N}}  e^{i(t-\tau)\partial_x^2}G(v_{0},w_{0},0) \, d\tau\quad \text{and} \quad \psi_{1} =e^{i\delta_{N}\partial_x^2}\psi_{0}.\]
In order to continue with the iteration, we need $\phi_{1}\in L^2$ and $\psi_{1}\in M^{p_{0},p_{0}'}$  satisfying \eqref{dp}. Therefore, we can obtain a solution to the I.V.P. corresponding to the initial value $\phi_{1}+\psi_{1}$ on $t\in [\delta_{N},2\delta_{N}]$,
given as
\begin{eqnarray}\label{ivp2iter}
\begin{cases}
 i \partial_t u +u_{xx} \pm |u|^{\alpha-1}u=0 \\
u(\delta_{N})=\phi_{1}+\psi_{1}.
\end{cases}        
    \end{eqnarray}
Based on \eqref{Mpp1}, we have
$\psi_{1} \in M^{p_{0},p_{0}'} $. Specifically, we have
$$\|\psi_{1}\|_{M^{p_{0},p_{0}'}} \leq \sup_{t\in [0,\delta_{N}]}\|e^{i t \partial_x^2}\psi_{0}\|_{M^{p_{0},p_{0}'}} \lesssim \|\psi_{0}\|_{M^{p_{0},p_{0}'}} \lesssim 1/N.$$  
We now proceed to show that $\|\phi_{1}\|_{L^{2}} \lesssim N^{\gamma}$ using the following proposition.  
\begin{prop}\label{winfty2} Let $v_{0}$ be the global solution with initial value $\phi_{0} \in L^2,\; w_{0}$ as defined in \eqref{w01} and $ \psi_{0}\in M^{p_{0},p_{0}'}$. Then, for any $N>1$, we have
$$\sup_{[0,\delta_{N}]}\|w_{0}-e^{i t \partial_x^2}\psi_{0}\|_{L^{2}}\lesssim N^{-1+\frac{2(\alpha-1)}{5-\alpha}\left(\frac{1}{p_{0}}-\frac{1}{2}\right)\gamma}. $$
\end{prop}
\begin{proof}
Using Lemma \ref{lemlwp2}, \eqref{gwpl2}, \eqref{dp} and Proposition \ref{w0exist}, we have
\begin{align*}
\|w_{0}-e^{i t \partial_x^2}\psi_{0}\|_{L^{\infty}_{\delta_{N}}L^{2}} = & \left|\left|\int_{0}^{t} e^{i(t-\tau)\partial_x^2}G(v_{0}, w_{0},0) \, d\tau \right|\right|_{L^{\infty}_{\delta_{N}}L^{2}}\\
\leq & \delta_{N}^{1-\frac{\alpha-1}{4}-\frac{1}{2}\left(\frac{1}{p_{0}}-\frac{1}{2}\right)} \|v_{0}\|_{Z(\delta_{N})}^{\alpha-1}\|w_{0}\|_{Y(\delta_{N})} + \delta_{N}^{1-\frac{\alpha-1}{2p_{0}}-\frac{1}{2}\left(\frac{1}{p_{0}}-\frac{1}{2}\right)}\|w_{0}\|_{Y(\delta_{N})}^{\alpha}\\
\lesssim & N^{-1+\frac{2(\alpha-1)}{5-\alpha}\left(\frac{1}{p_{0}}-\frac{1}{2}\right)\gamma} .
\end{align*}
\end{proof}
Taking into account \eqref{gwpl2}, \eqref{dp} and Proposition \ref{winfty2}, we have
\begin{equation*}
  \|\phi_{1}\|_{L^2} \lesssim N^{\gamma}+N^{-1+\frac{2(\alpha-1)}{5-\alpha}(\frac{1}{p_{0}}-\frac{1}{2})\gamma}. 
\end{equation*}
For $N$ sufficiently large and $(\alpha-1) /2 <p_{0}$, the right-hand side of the above inequality can be further bounded as follows: $$N^{\gamma}+N^{-1+\frac{2(\alpha-1)}{5-\alpha}\left(\frac{1}{p_{0}}-\frac{1}{2}\right)\gamma} \lesssim 2N^{\gamma}.$$
Consequently, the I.V.P. \eqref{ivp2iter}
has a local solution in $ [\delta_{N}, 2\delta_{N}].$
More generally, 
We can extend the time interval by using a similar iterative scheme. 
Let $K$ be the maximal number of iterations possible. We introduce an iterative initial data decomposition as $\phi_{k}+\psi_k$ for $1 \leq k \leq K$ (for $k=0,\; \phi_{0}+\psi_0$ are defined in \eqref{dp}) where
    \begin{eqnarray}\label{newidk}
    \begin{cases}
 \phi_{k}= v_{k-1}(k\delta_{N}) \pm i \displaystyle \int_{0}^{k\delta_{N}} e^{i(t-\tau) \partial_x^2}G(v_{k-1},w_{k-1},0) d\tau,\\
\psi_{k}=e^{ik\delta_{N}\partial_x^2}\psi_{0}. 
 \end{cases}
 \end{eqnarray}
Solving \eqref{NLS} with  $\phi_{k}\in L^{2}$ yields a solution $v_{k}$ (as described in \eqref{gwpl2} for $k=0$). Further, solving modified NLS \eqref{ivpMod}  with initial value $\psi_{k}$ and merging $v_{k},$ we get a solution to \eqref{NLS} in the interval $[k\delta_{N},(k+1)\delta_{N}].$\\
We repeat the iteration by redefining $\phi_{k+1}$ and $\psi_{k+1},$ and maintaining control of the nonlinear interaction term $(w_{k}((k+1)\delta_{N})-e^{i(k+1)\delta_{N}\partial_x^2}\psi_{0})$ as done in Proposition \ref{winfty2} for $k=0$ case, so that it can be absorbed into $v_{k}((k+1)\delta_{N})$ without losing  mass conservation of $\phi_{k+1}.$\\
The next goal is to determine a bound on $ K$. For that, we check for the bound of $\phi_{K}$ and $\psi_{K}$ in $L^{2}$ and $M^{p_{0},p_{0}'}$, 
 respectively.
 Observe that, using \eqref{Mpp1}, for any $k,$ we have 
 \begin{equation*}
\|\psi_{k}\|_{M^{p_{0},p_{0}'}}=\|e^{ik\delta_{N}\partial_x^2}\psi_{0}\|_{M^{p_{0},p_{0}'}}\lesssim \sup_{t\in [0,k\delta_{N}]}\|e^{i t \partial_x^2}\psi_{0}\|_{M^{p_{0},p_{0}'}} \lesssim \|\psi_{0}\|_{M^{p_{0},p_{0}'}} \lesssim 1/N.
 \end{equation*}
 Using \eqref{gwpl2} and Proposition \ref{winfty2} for the $K$th iteration, we have
\begin{equation}\label{KL2}
    \|\phi_{K}\|_{L^{2}} \lesssim N^{\gamma}+KN^{-1+\frac{2(\alpha-1)}{5-\alpha}(\frac{1}{p_{0}}-\frac{1}{2})\gamma}.
\end{equation}
  Note that the solution can be extended to the $K$th iteration until the $L^2$-norm of the Duhamel term accumulates up to the size $\sim N^{\gamma}$. That is
$$N^{\gamma}+KN^{-1+\frac{2(\alpha-1)}{5-\alpha}(\frac{1}{p_{0}}-\frac{1}{2})\gamma} \lesssim 2 N^{\gamma}.$$
 From the above inequality, we can derive the maximum number of iterations possible. Consequently, one can establish a local solution of I.V.P. \eqref{NLS} to the time given as 
 \begin{equation}\label{TN3}
    T_{N}=K\delta_{N}=CN^{1+\frac{2}{5-\alpha}\left(4-2\alpha-\frac{\alpha-1}{p_{0}}\right)\gamma}.
 \end{equation} 
 Thus, for any sufficiently large $N>1$, there is a local solution, denoted by $u_{N}$, of \eqref{NLS} well defined on the time interval $[0, T_N]$ for each $N$. 
 \textbf{This concludes the proof of Theorem \ref{mr} \eqref{lwp}}.\\

Since $N$ can be chosen to be arbitrarily large and in the case the exponent of $N$ in \eqref{TN3} is positive, we can obtain a global solution. Thus, we impose the following condition on $p_{0}$ given as
\begin{equation}\label{assumption}
      \frac{2}{5-\alpha} \big(4-2\alpha-\frac{\alpha-1}{p_{0}} \big)\gamma >-1.
  \end{equation}

 Substituting the lower bound of $p_{0}$ as given in \eqref{pminloc}, in \eqref{assumption}, we have an upper bound on $\gamma$ given as
 \begin{gather}\label{gammabound}
    \gamma_{\max}=\begin{cases}
    \infty &\text{if}\quad \alpha \in (1,\frac{3+\sqrt{57}}{6})\\
     \frac{(\alpha+1)(5-\alpha)}{2(3\alpha^2-3\alpha-4)}\quad&\text{if}\quad \alpha \in [\frac{3+\sqrt{57}}{6},3)\\
     \frac{2(5-\alpha)}{11\alpha-19}\quad&\text{if}\quad \alpha \in [3,\frac{10}{3}]\\
    \frac{\alpha(5-\alpha)}{4\alpha^2-3\alpha-5)}\quad&\text{if}\quad \alpha \in (\frac{10}{3},5).
    \end{cases}
\end{gather}
Observe that $\gamma$ as given in \eqref{betap} is a decreasing function of $p$. Hence, the condition $\gamma \in (0,\gamma_{\max})$
is equivalent to the range of $p \in (p_{\min},2)$. Hence, $p_{\min}$ can be obtained by substituting $\gamma_{\max}$
into the expression \eqref{betap} as 
\begin{gather}\label{pmin}p_{\min}=
    \begin{cases}
     \frac{\alpha+1}{\alpha} \quad&\text{if}\quad \alpha \in (1,\frac{3+\sqrt{57}}{6})\\
     \frac{5\alpha+3}{2(\alpha+2)} \quad&\text{if}\quad \alpha \in [\frac{3+\sqrt{57}}{6},3)\\
     \frac{9(\alpha-1)}{2(2\alpha-1)}\quad&\text{if}\quad \alpha \in [3,\frac{10}{3}]\\
    \frac{(\alpha-1)(3\alpha+5)}{2(\alpha^2-2\alpha+5)}\quad&\text{if}\quad \alpha \in (\frac{10}{3},5).
    \end{cases}
    \end{gather} 
This completes the proof of Theorem \ref{mr} \eqref{gwp}.
\qed

\subsection{Proof of Corollary \ref{globalprnls}}
 We first establish the following proposition.
\begin{prop}
[uniqueness]\label{uniqueNLS}
   Let $u_{1}$ be the local solution of \eqref{NLS} obtained in Theorem \ref{betterregularityNLS} and $u_{2}$ be the global solution of \eqref{NLS} obtained in Theorem \ref{mr}. Assume $u_{1}(0)=u_{2}(0):= u_{0} \in M^{p,p'}$ with $p$ satisfying \eqref{pminglobal}. Then, for $T_{1} \in (0,T^{*})$, we have $$u_{1}(x,t)=u_{2}(x,t),\quad \forall  (x,t)\in \R \times  [0, T_{1}].$$ 
\end{prop}
\begin{proof}
By Theorem \ref{betterregularityNLS} \footnote{Theorem \ref{mr} imposes a stronger restriction on the exponent $p$ than Theorem \ref{betterregularityNLS}. Therefore, by assuming $p$ satisfying \eqref{pminglobal}, it automatically satisfies \eqref{ppor}.}, there exist $T_1 \in (0,T^{*})$ such that  
\[u_1\in C_{T_1}M^{p,p'}\cap L^{Q_p(r)}_{T_1}L^r.\]
As a result of Theorem \ref{mr}, for any $T>0$,
\[u_2\in \left( L^{\infty}_{T}M^{p_0,p_0'}\cap L^{Q_{p_0}(r)}_{T}L^r\right) + \left( L^{\infty}_{T}L^2\cap L^{Q_{p_0}(r)}_{T}L^r\right) .\] 
Notice that  we cannot compare $M^{p,p'}$ and $M^{p_0, p'_0}$ spaces.
Since $p_{0}<p$, we have  $$Q_{p_0}(r) < Q_p(r)$$ 
This enables us to compare generalised Strichartz spaces. And, we have
\[L^{Q_p(r)}_{T_{1}}L^r \subset L^{Q_{p_0}(r)}_{T_{1}}L^r.\]
Hence, both solutions  $u_1, u_2 \in L^{Q_{p_0}(r)}_{T_{1}}L^r=:Y(T_{1}).$ 
By Duhamel’s formula \eqref{DuhamelP} and Lemma \ref{lemlwp}, for $T_{0}\in (0,T_{1}]$, we obtain
\begin{align*}
 \|u_{1}-u_{2}\|_{Y(T_{0})}&=\left|\left|\int_{0}^{t} e^{i(t-\tau)\partial_{x}^{2}}(|u_{1}|^{\alpha-1}u_{1}-|u_{2}|^{\alpha-1}u_{2})(\tau)d\tau \right|\right|_{Y(T_{0})} \\
     & \lesssim   T_{0}^{1-\frac{\alpha-1}{2p_{0}}} \|u_{1}-u_{2}\|_{Y(T_{0})}
   \left(\|u_{1}\|_{Y(T_{1})}^{\alpha-1} + \|u_{2}\|_{Y(T_{1})}^{\alpha-1}\right)\\
   &= 2 (\eta_{T_{1}})^{\alpha-1} T_{0}^{1-\frac{\alpha-1}{2p_{0}}} \|u_{1}-u_{2}\|_{Y(T_{0})}
    \end{align*}
where $\eta_{T_{1}}:=\max \{\|u_{1}\|_{Y(T_{1})}, \|u_{2}\|_{Y(T_{1})}\} $. 
Now, taking $T_{0}>0$ sufficiently small so that
  \begin{equation}\label{necessaryNLS}
      2 (\eta_{T_{1}})^{\alpha-1} T_{0}^{1-\frac{\alpha-1}{2p_{0}}}  \leq \frac{1}{2},
   \end{equation}   
we have $$\|u_{1}-u_{2}\|_{Y(T_{0})} \leq \frac{1}{2}\|u_{1}-u_{2}\|_{Y(T_{0})}.$$
This implies,   $$u_{1}=u_{2},\quad \forall (x,t)\in \R \times  [0,T_{0}].$$
It follows from the above difference estimate that $u_1(t) = u_2(t)$ on $t \in [0,T_{0}]$ for sufficiently small
$T_{0} \in (0,T_{1}]$. In a similar manner, we can observe
that  $u_1(t) = u_2(t)$ for all $t\in [T_{0},2T_{0}]$, since the time interval $T_{0}$ can be determined depending only on $\eta_{T_{1}}$, given as \eqref{necessaryNLS}. Repeating this continuity argument, we can establish uniqueness on $[0, T_{1}]$. This concludes the proof.
\end{proof}

\begin{proof}[\textbf{Proof of Corollary \ref{globalprnls}}]
Assume $u_{0} \in M^{p,p'}$ with $p$ satisfying \eqref{pminglobal}. By the blow-up alternative in Theorem \ref{betterregularityNLS}, extending the local solution to a global one is equivalent to proving that, for all finite $T>0$,
\begin{equation}\label{finiteblowup2}
    \sup_{0 \leq t < T} \|u(t)\|_{M^{p,p'}} < \infty.
\end{equation}
In Proposition \ref{uniqueNLS}, we have shown that the solution obtained in Theorem \ref{betterregularityNLS} and Theorem \ref{mr} equals almost everywhere. We denote it by $ u$.\\
Considering Theorem \ref{mr}, for any arbitrary $T>0$, \eqref{NLS} admits a global solution
\begin{align}
   \vspace{0.2cm} u \in & \left(L^{\infty}_{T} L^2\cap L^{Q_{p_{0}}(r)}_{T} L^{r}\right) +\left(L^{\infty}_{T}M^{p_{0},p_{0}'} \cap L^{Q_{p_{0}}(r)}_{T} L^{r}\right) \nonumber\\ 
    & \subset L^{Q_{p_{0}}(r)}_{T}L^{r} \subset L^{\frac{8r}{3r-4}}_{T} L^{r}.\label{essential}
     \end{align} 
The last embedding holds good since $4/3<p_{0}$. Denote $$Y_{2}(T):=L^{\frac{8r}{3r-4}}_{T} L^{r}.$$
Considering \eqref{Mpp1} for the homogeneous part of $u$, it is enough to estimate the $M^{p,p'}$ norm of the nonlinear term of $u$, (cf. \eqref{DuhamelP}) to claim \eqref{finiteblowup2}.\\
Next, we bound the $M^{p,p'}$-norm of the Duhamel operator, in terms of $Y_{2}(T)$, taking advantage of  \eqref{essential}. 
To this end, define $b=b(r)$ as
    \begin{equation}\label{B2}
        \frac{1}{b}=\frac{3}{2}-\frac{\alpha}{2r}-\frac{1}{2p}=\frac{12p-4-3\alpha p}{8p}+\frac{\alpha(3r-4)}{8r}.
      \end{equation}
      Taking into account the hypothesis on $p$ in Theorem \ref{betterregularityNLS}, as discussed in Remark \ref{why1}, and $p$ satisfies \eqref{pminglobal}, it follows that $$(b',(r/\alpha)') \in \widehat{\mathcal{X}}(p).$$ 
      Following the similar arguments as to derive  \eqref{integral21}, invoking \eqref{Mpp1}, \eqref{dsembeddings} and \eqref{weightest} with $(\sigma',\rho')=(b',(r/\alpha)')\in \widehat{\mathcal{X}}(p)$ with $b$ given as \eqref{B2} and H\"older's inequality with H\"older condition \eqref{B2}, to obtain
   \begin{align}
    \left|\left| \int_{0}^{t} e^{i(t-\tau)\partial_{x}^{2}}(|u|^{\alpha-1}u)(\tau)d\tau \right|\right|_{L_{T}^{\infty}M^{p,p'}}
    &\lesssim  \left|\left| \int_{0}^{t} e^{-i\tau\partial_{x}^{2}}(|u|^{\alpha-1}u)(\tau)d\tau \right|\right|_{L_{T}^{\infty}L^{p}} \nonumber\\
    &\lesssim \||\tau|^{\frac{1}{p}-\frac{1}{2}}|u(\tau)|^{\alpha}\|_{L_{T}^{b}L^{\frac{r}{\alpha}}}\nonumber\\
     &\lesssim \||\tau|^{\frac{1}{p}-\frac{1}{2}}\|_{L_{T}^{\frac{8p}{12p-4-3\alpha p}}}\|u\|_{Y_{2}(T)}^{\alpha}\nonumber\\
   &\lesssim \label{connetions} T^{1+\frac{1}{2p}-\frac{3\alpha}{8}}\|u\|_{Y_{2}(T)}^{\alpha}.
    \end{align}    
Note that the exponent of $T$ in \eqref{connetions} is positive since $p<2$ and $\alpha \in (1,10/3]$.  Thus, we see that $u(t)\in M^{p,p'}$ for any $t\in [0,T]$. Since $T>0$ is arbitrary, we have
$u\in L^{\infty}_{loc}(\R,M^{p,p'})$, and the result follows.
\end{proof}
\begin{Remark}\label{why10/3}
Observe that we have the same $b$ in the proof of Theorem \ref{betterregularityNLS} and Corollary \ref{globalprnls}. But, we bound the Duhamel operator in $M^{p,p'}$-norm by different spaces, depending on our requirements. Particularily, by $L^{Q_{p}(r)}_{T}L^{r}-$ and $Y_{2}(T)-$ norm, cf. \eqref{integral21} and \eqref{connetions}, respectively. This is well captured by the restriction on $\alpha$. Although Theorem \ref{betterregularityNLS} holds for $\alpha \in (1,5)$, Corollary \ref{globalprnls} is valid only for $\alpha \in (1,10/3]$. In view of the fact that the exponent of $T$ in \eqref{connetions} is not positive for $\alpha \in (10/3,5)$.
\end{Remark}

\section{Proof of Theorems \ref{betterregularityH}, \ref{mrh} and Corollary \ref{globalprh}}\label{Hresults}
\subsection{Proof of Theorem \ref{betterregularityH}}
Assume $u_{0} \in M^{s,s'},$ where $s$ satisfies \eqref{sbetterregularity}.
   Recall $r_{\nu}$ given as \eqref{rnu}. Define
\begin{equation*}
    X_{3}(T):=L^{Q_{s}(r_{\nu})}_{T}L^{r_{\nu}}, \quad (Q_{s}(r_{\nu}),r_{\nu}) \in \widehat{\mathcal{X}}(s),
\end{equation*}
and $$X'(T):=X_{3}(T) \cap L^{\infty}_{T}M^{s,s'}$$ 
    equipped with the norm 
$$\|v\|_{X'(T)}=\max \left\{\|v\|_{X_{3}(T)}, \|v\|_{L^{\infty}_{T}M^{s,s'}}\right\}. $$
Based on the Duhamel principle, the solutions to \eqref{NLSH} are equivalent to the integral equation 
\begin{equation}\label{duhamelH}
u(t)=e^{it\partial_{x}^{2}}u_{0}\pm i  \int_{0}^{t}  e^{i(t-\tau)\partial_{x}^{2}} \mathcal{H}_{\nu}(u)(\tau)\, d\tau:=\Lambda_{2}(u)(t). 
\end{equation}
Let $A$ and $T$ be positive real numbers (to be chosen later)   and define $$B_{A}(T):=\{u\in X'(T):\|u\|_{X'(T)}\leq A\}.$$
We will show that $\Lambda_{2}$ 
    is a contraction map on $B_{A}(T).$ 
    Firstly, we consider the linear evolution of $u_{0}\in M^{s,s'}$. Assume w.l.o.g. that $T\leq 1$. Using \eqref{gse3} and \eqref{embeddings}, we have
    \begin{equation}\label{linHartree1}
         \|e^{i t \partial_x^2}u_{0}\|_{X_{3}(T)} \lesssim 
        \|u_{0}\|_{M^{s,s'}}.
    \end{equation}
    Combining \eqref{linHartree1} and \eqref{Mpp1}, we have
    \begin{equation}\label{linHartreefinal}
        \|e^{i t \partial_x^2}u_{0}\|_{X'(T)} \lesssim \|u_{0}\|_{M^{s,s'}}.
         \end{equation}
    This suggests the choice of
    $A=C\|u_{0}\|_{M^{s,s'}}.$
    Using Lemma \ref{lemlwph} \eqref{xxx}, we have
    \begin{align}\label{iphartree1}
    \left|\left| \int_{0}^{t} e^{i(t-\tau)\partial_{x}^{2}}(\mathcal{H}_{\nu}(u))(\tau)d\tau \right|\right|_{X_{3}(T)}
    \lesssim   T^{1-\frac{\nu}{2}-\left(\frac{1}{s}-\frac{1}{2}\right)}\|u\|_{X_{3}(T)}^{3}.
    \end{align}
    Consider
   \begin{align}
       \frac{1}{\varphi}&=1+\frac{1}{2r_{\nu}}-\frac{1}{2s}=1+\frac{2}{r_{\nu}}-\frac{2}{s}+\frac{3}{Q_{s}(r_{\nu})},\label{HC3} \\
   \frac{1}{r'_{\nu}}&=\frac{1}{r_{\nu}}+\frac{\nu}{2}=\frac{1}{r_{\nu}}+\frac{2}{r_{\nu}}+\nu-1 \label{HC4}.
   \end{align}
    Now, using \eqref{Mpp1}, \eqref{dsembeddings}, \eqref{weightest} with $(\sigma',\rho')=(\varphi',(4/(2+\nu))' )\in \widehat{\mathcal{X}}(s)$, H\"older's inequality with H\"older conditions \eqref{HC3} and \eqref{HC4} with
respect to space and time variables, respectively, and the Hardy-Littlewood Sobolev inequality, we have
   \begin{align}
    \left|\left| \int_{0}^{t} e^{i(t-\tau)\partial_{x}^{2}}\mathcal{H}_{\nu}(u)(\tau)d\tau \right|\right|_{L_{T}^{\infty}M^{s,s'}}
   &= \left|\left| e^{it\partial_{x}^{2}} \left\{\int_{0}^{t} e^{-i\tau\partial_{x}^{2}}\mathcal{H}_{\nu}(u)(\tau)d\tau \right\}\right|\right|_{L_{T}^{\infty}M^{s,s'}}\nonumber\\
   &\lesssim  \left|\left| \int_{0}^{t} e^{-i\tau\partial_{x}^{2}}\mathcal{H}_{\nu}(u)(\tau)d\tau \right|\right|_{L_{T}^{\infty}M^{s,s'}} \nonumber\\
    &\lesssim  \left|\left| \int_{0}^{t} e^{-i\tau\partial_{x}^{2}}\mathcal{H}_{\nu}(u)(\tau)d\tau \right|\right|_{L_{T}^{\infty}L^{s}} \nonumber\\
    &\lesssim \||\tau|^{\frac{1}{s}-\frac{1}{2}}\mathcal{H}_{\nu}(u)(\tau)\|_{L_{T}^{\varphi}L^{\frac{4}{2+\nu}}}  \nonumber\\
    &= \||\tau|^{\frac{1}{s}-\frac{1}{2}} \|\mathcal{H}_{\nu}(u)(\tau)\|_{L^{\frac{4}{2+\nu}}}\|_{L_{T}^{\varphi}} \nonumber\\
    &\lesssim \left|\left| |\tau|^{\frac{1}{s}-\frac{1}{2}} \| |\cdot|^{-\nu}\ast |u|^{2}\|_{L^{\frac{2}{\nu}}}\| u\|_{L^{\frac{4}{2-\nu}}}\right|\right|_{L_{T}^{\varphi}} \nonumber\\
     &\lesssim \left|\left| |\tau|^{\frac{1}{s}-\frac{1}{2}} \| u\|^{3}_{L^{\frac{4}{2-\nu}}}\right|\right|_{L_{T}^{\varphi}} \nonumber\\
      &\lesssim \| |\tau|^{\frac{1}{s}-\frac{1}{2}}\|_{L_{T}^{\frac{2s}{(4-\nu)s-4 }}} \|u\|_{X_{3}(T)}^{3} \nonumber\\
    &\lesssim T^{1-\frac{\nu}{2}-\left(\frac{1}{s}-\frac{1}{2}\right)}\|u\|_{X_{3}(T)}^{3}.\label{iphartree12}
    \end{align}
    
    Thus, combining \eqref{iphartree1} and \eqref{iphartree12}, we have
    \begin{equation}\label{iphartreefinal}
          \left|\left| \int_{0}^{t} e^{i(t-\tau)\partial_x^2}(\mathcal{H}_{\nu}(u)(\tau)d\tau \right|\right|_{X'(T)} \lesssim  T^{1-\frac{\nu}{2}-\left(\frac{1}{s}-\frac{1}{2}\right)}\|u\|_{X'(T)}^{3}.
    \end{equation}
    Taking
    \begin{equation}\label{T*H}
        T :=\min \left\{ 1,~C \|u_{0}\|^{-\frac{2}{1-\frac{\nu}{2}-\left(\frac{1}{s}-\frac{1}{2}\right)}}_{M^{s,s'}} \right \},
    \end{equation}
 we combine \eqref{linHartreefinal} and \eqref{iphartreefinal} to conclude $\Lambda_{2}(u)\in B_{A}(T).$\\
 Using the previous argument employed for $u,v \in B_{A}(T)$, we can similarly establish
 $$\|\Lambda_{2}(u)-\Lambda_{2}(v)\|_{X'(T)}\leq \frac{1}{2}\|u-v\|_{X'(T)}.$$
Thus, by the contraction mapping theorem,  we obtain a unique fixed point for $\Lambda_{2}$ which is a solution to \eqref{NLSH}.
\subsection{Proof of Theorem \ref{mrh} }
Assume that
$u_{0} \in M^{s,s'},$ where $s \in (s_{0},2]$. 
Using Lemma \ref{ipt},  for any $N>1$ and $u_0 \in M^{s,s'},$ there exist $\phi_0 \in L^2$ and $ \psi_0 \in M^{s_{0},s_{0}'}$ 
such that 
\begin{gather}\label{dph}
    \begin{cases}
    u_0= \phi_0 + \psi_0\\
\|\phi_0\|_{L^2} \lesssim N^{\tilde{\alpha}},  \\
\|\psi_0\|_{M^{s_{0},s_{0}'}} \lesssim  \frac{1}{N}
\end{cases}
\end{gather}
where \begin{equation}\label{alphap}
    \tilde{\alpha} =\frac{\frac{1}{s}-\frac{1}{2}}{\frac{1}{s_{0}}-\frac{1}{s}}.
\end{equation}
We seek a solution $u$ to \eqref{NLSH} of the form
$$u:=v_{0}(t)+w_{0}(t).$$
Here, $v_{0}(t)$ satisfies the I.V.P.
\begin{eqnarray}\label{ivpL2H}
\begin{cases}
 i \partial_t v_{0} +(v_{0})_{xx}\pm \mathcal{H}_{\nu}(v_{0})=0 \\
v_{0}(\cdot,0)=\phi_{0}\in L^2.  
\end{cases}        
    \end{eqnarray}
While $w_{0}(t)$ satisfies the I.V.P.
\begin{eqnarray}\label{ivpModH}
\begin{cases}
 i \partial_t w_{0} +(w_{0})_{xx} \pm (\mathcal{H}_{\nu}(v_{0}+w_{0})-\mathcal{H}_{\nu}(v_{0}))=0 \\
w_{0}(\cdot,0)=\psi_{0}\in M^{s_{0},s_{0}'} .
\end{cases}
\end{eqnarray}
We first discuss the solution of \eqref{ivpL2H}. We recall that \eqref{ivpL2H} has a unique global solution $v_{0}$ in the space, cf. \cite[Proposition 3.4]{DGBJDE}, 
\begin{eqnarray}\label{gwpl2h}
\begin{cases}
  v_{0} \in C(\R, L^2) \cap L_{loc}^{q(r_{\nu})}(\R,L^{r_{\nu}} ), \quad (q(r_{\nu}),r_{\nu}) \in \mathcal{A}\\
\sup_{(q(r),r)\in \mathcal{A}}\|v_{0}\|_{L_{loc}^{q(r_{\nu})} L^{r_{\nu}}} \lesssim \|\phi_{0}\|_{L^{2}}.
\end{cases}
\end{eqnarray}
Note that since $v_{0} \in L_{loc}^{q(r_{\nu})}(\R,L^{r_{\nu}})$,  we also have $ v_{0}\in L_{loc}^{Q_{s_{0}}(r_{\nu})}(\R,L^{r_{\nu}})$ by using the inclusion relation for the Lebesgue spaces when $s_{0}<2$ on finite measure spaces, also see Remark \ref{chipAp}.\\
Since $v_0 $ is globally defined in appropriate spaces as a result of  \eqref{gwpl2h}, to establish the desired global existence, we need to examine the time interval of existence of the solution to the I.V.P. \eqref{ivpModH}, denoted as $w_{0}$.
We start by proving that a local solution to \eqref{ivpModH} exists in an interval of time $[0,\delta_{N}]$ defined as
\begin{equation}\label{TNh}
    \delta_{N}\sim (N^{\tilde{\alpha}})^{-\frac{4}{2-\nu}}.
\end{equation}
 Recall the space $X_{1}(T)$ (and $X_{2}(T)$) as defined in \eqref{X1(T)} (and \eqref{X2(T)}). 
 We seek a solution to \eqref{ivpModH} by considering its equivalent integral formulation,
 \begin{equation}\label{modsoln1}
 w_{0}:=e^{it\partial_{x}^{2}}\psi_{0} \pm i \int_{0}^{t}  e^{i(t-\tau)\partial_{x}^{2}} \left\{\mathcal{H}_{\nu}(v_{0}+w_{0}) - \mathcal{H}_{\nu}(v_{0}) \right\} \, d\tau
\end{equation}
in the complete metric space $B_{R}$, for some $R>0$, defined as
$$B_{R}:=\{w_{0} \in X_{1}(\delta_{N}):\|w_{0}\|_{X_{1}(\delta_{N})}\leq R/N \}.$$   
For this, we claim that the operator
    $$\Lambda(w_{0}):=e^{it\partial_{x}^{2}}\psi_{0}\pm i \int_{0}^{t}  e^{i(t-\tau)\partial_{x}^{2}} \left\{\mathcal{H}_{\nu}(v_{0}+w_{0}) - \mathcal{H}_{\nu}(v_{0}) \right\} \, d\tau$$
    is well defined on $B_{R}$ and is a contraction map from $B_{R}$ to itself.
    Assume $w_{0} \in B_{R}$. Taking \eqref{gse3}, \eqref{embeddings} and \eqref{dph} into account, we have
    \begin{equation}\label{est1}
        \|e^{it\partial_{x}^{2}}\psi_{0}\|_{X_{1}(\delta_{N})} \lesssim 1/N.
    \end{equation}
    Next, we use Proposition \ref{gse2} to obtain
    \begin{align*}
         \left|\left| \int_{0}^{t}  e^{i(t-\tau)\partial_{x}^{2}} \left\{\mathcal{H}_{\nu}(v_{0}+w_{0}) - \mathcal{H}_{\nu}(v_{0}) \right\} \, d\tau \right|\right|_{X_{1}(\delta_{N})}& \lesssim \left|\left| \mathcal{H}_{\nu}(v_{0}+w_{0}) - \mathcal{H}_{\nu}(v_{0})\right|\right|_{L^{\sigma}L^{r_{\nu}'}} 
\end{align*}
where $\sigma$ can be determined by the relation $$2+\frac{1}{s_{0}}=\frac{2}{\sigma}+\frac{1}{r_{\nu}'}.$$ 
It is enough to estimate the terms 
$$\mathcal{H}_{\nu}(v_{0},v_{0},w_{0}),\quad \mathcal{H}_{\nu}(w_{0},w_{0},w_{0}),\quad \mathcal{H}_{\nu}(v_{0},w_{0},w_{0})\quad
\text{for}\quad v_{0} \in X_{2}(\delta_{N}), w_{0} \in X_{1}(\delta_{N}).$$
Using Lemma \ref{lemlwph} \eqref{yyx} and \eqref{dph}, we have 
\begin{align*}
\|\mathcal{H}_{\nu}(v_{0},v_{0},w_{0})\|_{L^{\sigma}L^{r_{\nu}'}} &\lesssim \delta_{N}^{1-\frac{\nu}{2}}\|v_{0}\|^{2}_{X_{2}(\delta_{N})}
\|w_{0}\|_{X_{1}(\delta_{N})}\\
&\lesssim (N^{\tilde{\alpha}})^{-2} \cdot N^{2\tilde{\alpha}} \cdot  \frac{R}{N} =C \frac{R}{N}
    \end{align*}
Similarly, applying Lemma \ref{lemlwph} \eqref{yxx} and \eqref{dph}, we get
\begin{align*}
\|\mathcal{H}_{\nu}(v_{0},w_{0},w_{0})\|_{L^{\sigma}L^{r'}} &\lesssim \delta_{N}^{1-\frac{\nu}{2}-\frac{1}{2}\left(\frac{1}{s_{0}}-\frac{1}{2}\right)}\|v_{0}\|_{X_{2}(\delta_{N})}
\|w_{0}\|^{2}_{X_{1}(\delta_{N})}\\
&\lesssim (N^{\tilde{\alpha}})^{-2} \cdot (N^{\tilde{\alpha}})^{\frac{2}{2-\nu}\left(\frac{1}{s_{0}}-\frac{1}{2}\right)} \cdot N^{\tilde{\alpha}} \cdot  \left(\frac{R}{N}\right)^2 =C \left(\frac{R}{N}\right)^2
    \end{align*}
Lastly, using Lemma \ref{lemlwph} \eqref{xxx} and \eqref{dph}, we obtain,
\begin{align*}
\|\mathcal{H}_{\nu}(v_{0},v_{0},w_{0})\|_{L^{\sigma}L^{r'}} &\lesssim \delta_{N}^{1-\frac{\nu}{2}-\left(\frac{1}{s_{0}}-\frac{1}{2}\right)}
\|w_{0}\|^{3}_{X_{1}(\delta_{N})}\\
&\lesssim (N^{\tilde{\alpha}})^{-2} \cdot (N^{\tilde{\alpha}})^{\frac{4}{2-\nu}\left(\frac{1}{s_{0}}-\frac{1}{2}\right)} \cdot  \left(\frac{R}{N}\right)^{3} =C \left(\frac{R}{N}\right)^{3}.
    \end{align*}
    The above inequalities hold since
    $$\frac{2}{2-\nu}\left(\frac{1}{s_{0}}-\frac{1}{2}\right) <1,$$ which follows from the assumption that  $$\frac{2}{3-\nu}<\frac{4}{2+\nu}<s_{0}\quad \text{and}\quad 0<\nu<1.$$
   Collecting these estimates and \eqref{est1}, we obtain
   \begin{align*}
       \|\Lambda(w_{0})\|_{X_{1}(\delta_{N})} \leq \frac{C_{1}}{N} +C_{2} \left(\frac{R}{N}+\left(\frac{R}{N}\right)^2+\left(\frac{R}{N}\right)^{3}\right)
   \end{align*}
   We choose $C_{1}$ and $C_{2}$ such that the right-hand side of the last inequality is less than or equal to $R/N$ with $N>N_{0}$ for some $N_{0}>1.$ Thus, $\Lambda(w_{0}):B_{R} \to B_{R} $ is a well-defined map. 
   Similarly, we can show that $\Lambda(w_{0})$ is a contraction map. Consequently, by the Banach fixed-point theorem, we get a local solution 
   $w_{0} \in X_{1}(\delta_{N})$\footnote{Also, $w_{0}-e^{it\partial_{x}^{2}}\psi_{0} \in X_{1}(\delta_{N})$ since $w_{0},\;e^{it\partial_{x}^{2}}\psi_{0} \in X_{1}(\delta_{N}).$}.\\
   Therefore, we have a local solution $u:=v_{0}+w_{0}$ to \eqref{NLSH} in $[0,\delta_{N}]$ such that  $$v_0 + (w_0-e^{it\partial_{x}^{2}} \psi_0) \in  L^{Q_{s_{0}}(r_{\nu})}_{\delta_{N}}L^{r_{\nu}}$$ and $$
    e^{it\partial_{x}^{2}} \psi_0 \in L^{\infty}_{\delta_{N}}M^{s_{0},s_{0}'} .$$
  Next, we aim to extend our solution to the interval $[\delta_{N}, 2\delta_{N}]$ by using a similar procedure but with the new initial data as the sum of the following two functions:    
\[ \phi_{1}=v_{0}(\delta_{N})+(w_{0}(\delta_{N})-e^{i\delta_{N}\partial_{x}^{2}}\psi_{0}) \quad \text{and} \quad \psi_{1} =e^{i\delta_{N}\partial_{x}^{2}}\psi_{0}.\]
satisfying \eqref{dph}.
Using \eqref{gwpl2h}, we have $v_{0}(\delta_{N}) \in L^2$ and by \eqref{Mpp1}, we have $e^{i\delta_{N}\partial_{x}^{2}}\psi_{0} \in M^{s_{0},s_{0}'}$. Specifically, we have
$$\|\psi_{1}\|_{M^{s_{0},s_{0}'}} \leq \sup_{t\in [0,\delta_{N}]}\|e^{it\partial_{x}^{2}}\psi_{0}\|_{M^{s_{0},s_{0}'}} \lesssim \|\psi_{0}\|_{M^{s_{0},s_{0}'}} \lesssim 1/N.$$  
It is important to observe that $w_{0}(\delta_{N})-e^{i\delta_{N}\partial_{x}^{2}}\psi_{0} \in L^{2}$. Hence, $\phi_{1} \in L^2$. To demonstrate this, we utilize Strichartz estimates (Proposition \ref{fst}) to obtain
\begin{align*}
    \sup_{t\in [0,\delta_{N}]}\|w_{0}(t)-e^{it\partial_{x}^{2}}\psi_{0}\|_{L^{2}} &\lesssim \left|\left| \mathcal{H}_{\nu}(v_{0}+w_{0}) - \mathcal{H}_{\nu}(v_{0})\right|\right|_{L^{\sigma}L^{\rho}} 
\end{align*}
with $(\sigma',\rho') \in \mathcal{A}.$ 
As seen before, it is enough to estimate $$\mathcal{H}_{\nu}(v_{0},v_{0},w_{0}),\quad \mathcal{H}_{\nu}(w_{0},w_{0},w_{0}),\quad \mathcal{H}_{\nu}(v_{0},w_{0},w_{0})\quad
\text{for}\quad v_{0} \in X_{2}(\delta_{N}), w_{0} \in X_{1}(\delta_{N}).$$ 
Using Lemma \eqref{lemlwph} \eqref{yyx} and \eqref{dph}, it can be estimated as
\begin{align*}
\|\mathcal{H}_{\nu}(v_{0},v_{0},w_{0})\|_{L^{\sigma}L^{\rho}} &\lesssim \delta_{N}^{1-\frac{\nu}{2}-\frac{1}{2}\left(\frac{1}{s_{0}}-\frac{1}{2}\right)}\|v_{0}\|^{2}_{X_{2}(\delta_{N})}
\|w_{0}\|_{X_{1}(\delta_{N})}\\
& \lesssim N^{-1+\frac{2}{2-\nu}\left(\frac{1}{s_{0}}-\frac{1}{2}\right)\tilde{\alpha}}.
    \end{align*}
Similarly, applying Lemma \ref{lemlwph} \eqref{yxx} and \eqref{dph}, we get
\begin{align*}
\|\mathcal{H}_{\nu}(v_{0},w_{0},w_{0})\|_{L^{\sigma}L^{\rho}} &\lesssim \delta_{N}^{1-\frac{\nu}{2}-\left(\frac{1}{s_{0}}-\frac{1}{2}\right)}\|v_{0}\|_{X_{2}(\delta_{N})}
\|w_{0}\|^{2}_{X_{1}(\delta_{N})}\\
& \lesssim N^{-2+\frac{4}{2-\nu}\left(\frac{1}{s_{0}}-\frac{1}{2}\right)\tilde{\alpha}-\tilde{\alpha}}
    \end{align*}
and, applying Lemma \ref{lemlwph} \eqref{xxx} and \eqref{dph}, 
\begin{align*}
\|\mathcal{H}_{\nu}(v_{0},v_{0},w_{0})\|_{L^{\sigma}L^{\rho}} &\lesssim \delta_{N}^{1-\frac{\nu}{2}-\frac{3}{2}\left(\frac{1}{s_{0}}-\frac{1}{2}\right)}
\|w_{0}\|^{3}_{X_{1}(\delta_{N})}\\
& \lesssim N^{-3+\frac{6}{2-\nu}\left(\frac{1}{s_{0}}-\frac{1}{2}\right)\tilde{\alpha}-2\tilde{\alpha}}.
    \end{align*}
The above inequalities hold good if 
$$\frac{6}{7-2\nu}<s_{0}$$
which follows since
$$\frac{4}{2+\nu}<s_{0} \quad \text{and} \quad 0<\nu<1.$$
Finally, for sufficiently large $N,$ we have
\begin{align}
    \sup_{t\in [0,\delta_{N}]}\|w_{0}(t)-e^{it\partial_{x}^{2}}\psi_{0}\|_{L^{2}} &\lesssim  N^{-1+\frac{2}{2-\nu}\left(\frac{1}{s_{0}}-\frac{1}{2}\right)\tilde{\alpha}} .\label{winfty2H}
\end{align}    
    Thus, by \eqref{mass}, \eqref{dph} and \eqref{winfty2H}, we have
    \begin{equation*}
        \|\phi_{1}\|_{L^2} \lesssim N^{\tilde{\alpha}}+N^{-1+\frac{2}{2-\nu}\left(\frac{1}{s_{0}}-\frac{1}{2}\right)\tilde{\alpha}}. 
    \end{equation*}
    In order to extend our solution by $\delta_{N}$ and apply Lemma \ref{ipt} for newly defined initial data decomposition, we need 
    $$N^{\tilde{\alpha}}+N^{-1+\frac{2}{2-\nu}\left(\frac{1}{s_{0}}-\frac{1}{2}\right)\tilde{\alpha}} \lesssim  2 N^{\tilde{\alpha}}.$$
For $N$ sufficiently large and $2/ (3-\nu)<s_{0}, $ we have 
$$N^{-1+\frac{2}{2-\nu}\left(\frac{1}{s_{0}}-\frac{1}{2}\right)\tilde{\alpha}} \lesssim N^{\tilde{\alpha}}  .$$
Thus, \eqref{NLSH} has a local solution in an interval of time $[\delta_{N},2\delta_{N}]$.
More generally, we 
 aim to extend our solution further by using a similar procedure. We define $\phi_{k}$ and $\psi_k$ for $k \geq 1$ (for $k=0,\; \phi_{0}$ and $\psi_0$ are defined in \eqref{dp}) as follows:
    \begin{eqnarray}
    \begin{aligned}\label{iter1}
 \phi_{k}&= v_{k-1}(k\delta_{N})  \pm i \displaystyle \int_{0}^{k\delta_{N}} e^{i(k\delta_{N}-\tau)\partial_{x}^{2}} \left\{\mathcal{H}_{\nu}(v_{k-1}+w_{k-1}) - \mathcal{H}_{\nu}(v_{k-1}) \right\} \,d\tau \quad \\
 \text{and} \quad \psi_{k}&= 
 e^{ik\delta_{N}\partial_{x}^{2}}\psi_{0}.
    \end{aligned}
       \end{eqnarray}
 We repeat the iteration by redefining $\phi_{k+1}$ and $\psi_{k+1},$ and maintaining control of the nonlinear interaction term $(w_{k}((k+1)\delta_{N})-e^{i(k+1)\delta_{N}\partial_{x}^{2}}\psi_{0})$ so that it can be absorbed into $v_{k}((k+1)\delta_{N})$ without losing  mass conservation of $\phi_{k+1}.$
 That is, we can obtain the solution by the above iteration scheme as long as $\psi_{K}$ and $\phi_{K}$ satisfy \eqref {dph}.
Observe that, using \eqref{Mpp1} and \eqref{dph}, for any $k,$ we have 
 \begin{equation*}
\|\psi_{k}\|_{M^{s_{0},s_{0}'}}=\|e^{ik\delta_{N}\partial_{x}^{2}}\psi_{0}\|_{M^{s_{0},s_{0}'}}\lesssim \sup_{t\in [0,k\delta_{N}]}\|e^{it\partial_{x}^{2}}\psi_{0}\|_{M^{s_{0},s_{0}'}} \lesssim \|\psi_{0}\|_{M^{s_{0},s_{0}'}} \lesssim 1/N.
 \end{equation*}
 And, $$\|\phi_{K}\|_{L^{2}} \leq CN^{\tilde{\alpha}}+K\delta_{N}^{-1+\frac{2}{2-\nu}\left(\frac{1}{s_{0}}-\frac{1}{2}\right)\tilde{\alpha}} \leq 2C N^{\tilde{\alpha}}$$
 from which we can extract the maximum number of iterations possible $ K$, as
 $$K \lesssim N^{1+\left[1-\frac{2}{2-\nu}\left(\frac{1}{s_{0}}-\frac{1}{2}\right)\right]\tilde{\alpha}}.$$
Consequently, we can establish a local solution of \eqref{NLSH} to the time $T_{N}:=K\delta_{N}$ given in \eqref{TNhar}. 
Since $N$ can be chosen to be arbitrarily large, and in the case where the exponent of $N$ is positive, we obtain a global solution. As a result, the assumption \eqref{assumptionhar} becomes necessary to get a global solution. Also, see Remark \ref{remarkhartee}\eqref{remarkhartee2}. This concludes the proof of Theorem \ref{mrh}.

\subsection{Proof of Corollary \ref{globalprh}}
Prior to proving Corollary \ref{globalprh}, we require the following proposition.
\begin{prop}[uniqueness]\label{uniqueH}
   Let $u_{1}$ be the local solution of \eqref{NLSH} obtained in Theorem \ref{betterregularityH} and $u_{2}$ be the global solution of \eqref{NLSH} obtained in Theorem \ref{mrh}. Assume $u_{1}(0)=u_{2}(0):=u_{0} \in M^{s,s'}$ for $s$ satisfying \eqref{s0gwphar}. Then, for $T' \in (0,T_{*})$, we have $$u_{1}(x,t)=u_{2}(x,t),\quad \forall  (x,t)\in \R \times  [0, T'].$$ 
\end{prop}
\begin{proof}
Recall by Theorem \ref{betterregularityH}, for $s$ satisfying \eqref{s0gwphar} \footnote{Note that if $s$ satisfies \eqref{s0gwphar}, then $s$ consequently satisfies \eqref{sbetterregularity}.}, there exist $T' \in (0,T_{*})$ such that 
\[u_1\in C_{T'}M^{s,s'}\cap L^{Q_s(r_{\nu})}_{T'}L^{r_{\nu}}.\]
As a result of Theorem \ref{mrh}, for any arbitrary $T>0$ and $s$ satisfying \eqref{s0gwphar}, \eqref{NLSH} possesses a solution  
\[u_2\in \left( L^{\infty}_{T}M^{s_0,s_0'}\cap L^{Q_{s_0}(r_{\nu})}_{T}L^{r_{\nu}}\right) + \left( L^{\infty}_{T}L^2\cap L^{Q_{s_0}(r_{\nu})}_{T}L^{r_{\nu}}\right) .\]
Since $s_{0}<s$, we have  $$Q_{s_0}(r_{\nu}) < Q_s(r_{\nu}).$$  
Hence,
\[L^{Q_s(r_{\nu})}_{T'}L^{r_{\nu}} \subset L^{Q_{s_0}(r_{\nu})}_{T'}L^{r_{\nu}}.\]
Thus, both solutions  $u_1, u_2 \in L^{Q_{s_0}(r_{\nu})}_{T'}L^r=:X_{1}(T')$.
By Duhamel’s formula \eqref{duhamelH} and Lemma \ref{lemlwph} \eqref{xxx}, for $T_{0}\in (0, T']$, we have
\begin{align*}
        \|u_{1}-u_{2}\|_{X_{1}(T_{0})}&=\left|\left|\int_{0}^{t} e^{i(t-\tau)\partial_{x}^{2}}(\mathcal{H}_{\nu}(u_{1})-\mathcal{H}_{\nu}(u_{2}))(\tau)d\tau \right|\right|_{X_{1}(T_{0})} \\
        & \lesssim   T_{0}^{1-\frac{\nu}{2}-\left(\frac{1}{s_{0}}-\frac{1}{2}\right)} \|u_{1}-u_{2}\|_{X_{1}(T_{0})}
   \left(\|u_{1}\|_{X_{1}(T')}^{2} + \|u_{2}\|_{X_{1}(T')}^{2}\right)\\
   &= 2 (\Tilde{\eta}_{T'})^{2} \; T_{0}^{1-\frac{\nu}{2}-\left(\frac{1}{s_{0}}-\frac{1}{2}\right)} \|u_{1}-u_{2}\|_{X_{1}(T_{0})},
    \end{align*}
    where $\Tilde{\eta}_{T'}:= \max\{\|u_{1}\|_{X_{1}(T')}, \|u_{2}\|_{X_{1}(T')}\}$.
  Now, taking $T_{0}>0$ sufficiently small such that
  $$ 2 (\Tilde{\eta}_{T'})^{2}\;  T_{0}^{1-\frac{\nu}{2}-\left(\frac{1}{s_{0}}-\frac{1}{2}\right)} \leq \frac{1}{2},$$
   we have $$\|u_{1}-u_{2}\|_{X_{1}(T_{0})} \leq \frac{1}{2}\|u_{1}-u_{2}\|_{X_{1}(T_{0})}.$$
Therefore,   $$u_{1}(x,t)=u_{2}(x,t),\quad \forall (x,t)\in \R \times  [0,T_{0}].$$
It follows from the above difference estimate that $u_1 = u_2$ on $[0,T_{0}]$ for sufficiently small
$T_{0} \in (0,T']$. Thus, we can establish uniqueness on $[0, T']$ by a continuity argument. This concludes the proof.
\end{proof}

\begin{proof}[\textbf{Proof of Corollary \ref{globalprh}}]
 Let $u(0)\in M^{s,s'}$ with $s$ satisfying \eqref{s0gwphar}. 
Using the blow-up alternative, the local solution from Theorem \ref{betterregularityH} can be extended globally if we establish that for all finite $T>0$,
 \begin{equation}\label{finiteblowup3}
    \sup_{0 \leq t < T} \|u(t)\|_{M^{s,s'}} < \infty.
\end{equation}
In Proposition \ref{uniqueH}, we have shown that the solution obtained in Theorem \ref{betterregularityH} and Theorem \ref{mrh} equals almost everywhere. We denote it by $ u$.\\
From Theorem \ref{mrh}, we have \begin{align}
   u \in & \left(L^{\infty}_{loc}(\R,L^2)\cap  L^{Q_{s_{0}}(r_{\nu})}_{loc}(\R,L^{r_{\nu}})\right) +\left(L^{\infty}_{loc}(\R,M^{s_{0},s_{0}'})\cap L^{Q_{s_{0}}(r_{\nu})}_{loc}(\R,L^{r_{\nu}})\right) \nonumber \\
    & \subset L^{Q_{s_{0}}(r_{\nu})}_{loc}(\R,L^{r_{\nu}})
    \subset L^{\frac{8}{1+\nu}}_{loc}(\R,L^{r_{\nu}})\label{essential2}. 
\end{align}
The above embedding holds since $4/3 < s_{0} < 2$. Denote $$X_{4}(T):=L^{\frac{8}{1+\nu}}_{T} L^{r_{\nu}}$$
Considering \eqref{Mpp1} for the linear part of $u$, it is enough to estimate the $M^{s,s'}$ norm of the Duhamel operator (cf. \eqref{duhamelH}) to claim \eqref{finiteblowup3}.
Denote 
\begin{align}\label{varphi1.11}
\frac{1}{\varphi}
    &:=1+\frac{1}{2r_{\nu}}-\frac{1}{2s}
    =\frac{(7-4\nu)s-4}{8s}+\frac{3(1+\nu)}{8},\\
    \frac{1}{r'_{\nu}}&=\frac{1}{r_{\nu}}+\frac{\nu}{2}=\frac{1}{r_{\nu}}+\frac{2}{r_{\nu}}+\nu-1 \label{HC5}.
   \end{align}
 Repeating the argument used in Theorem \ref{betterregularityH} to derive inequality \eqref{iphartree12}, using \eqref{Mpp1}, \eqref{dsembeddings} and \eqref{weightest} with $(\sigma',\rho')=(\varphi',(4/(2+\nu))' )\in \widehat{\mathcal{X}}(s)$ and $\varphi$ given as \eqref{varphi1.11}, H\"older's inequality with H\"older  conditions \eqref{varphi1.11} and \eqref{HC5} with
respect to space and time variables, respectively, along with the Hardy-Littlewood Sobolev inequality, we have
   \begin{align}
    \left|\left| \int_{0}^{t} e^{i(t-\tau)\partial_{x}^{2}}\mathcal{H}_{\nu}(u)(\tau)d\tau \right|\right|_{L_{T}^{\infty}M^{s,s'}}
    &\lesssim  \left|\left| \int_{0}^{t} e^{-i\tau\partial_{x}^{2}}\mathcal{H}_{\nu}(u)(\tau)d\tau \right|\right|_{L_{T}^{\infty}L^{s}} \nonumber\\
    &\lesssim \||\tau|^{\frac{1}{s}-\frac{1}{2}}\mathcal{H}_{\nu}(u)(\tau)\|_{L_{T}^{\varphi}L^{\frac{4}{2+\nu}}}  \nonumber\\
    &\lesssim \left|\left| |\tau|^{\frac{1}{s}-\frac{1}{2}} \| |\cdot|^{-\nu}\ast |u|^{2}\|_{L^{\frac{2}{\nu}}}\| u\|_{L^{r_{\nu}}}\right|\right|_{L_{T}^{\varphi}} \nonumber\\
     &\lesssim \left|\left| |\tau|^{\frac{1}{s}-\frac{1}{2}} \| u\|^{3}_{L^{r_{\nu}}}\right|\right|_{L_{T}^{\varphi}} \nonumber\\
      &\lesssim \| |\tau|^{\frac{1}{s}-\frac{1}{2}}\|_{L_{T}^{\frac{8s}{(7-4\nu)s-4 }}} \|u\|_{X_{4}(T)}^{3} \nonumber\\
    &\lesssim T^{\frac{3}{8}-\frac{\nu}{2}+\frac{1}{2s}}\|u\|_{X_{4}(T)}^{3}\label{TglobalHartree}.
    \end{align}
Note that the exponent of $T$ in \eqref{TglobalHartree} is positive since $s<2$ and $0<\nu<1$. Considering \eqref{essential2}, $u(t)\in M^{s,s'}$ for any $t\in [0,T]$. Since $T>0$ is arbitrary,
$u\in L^{\infty}_{loc}(\R,M^{s,s'})$.
\end{proof}

\section{Appendix}\label{app}
For the convenience of the reader, we briefly give the proof of \eqref{dp}: 
 \begin{lem}[interpolation]\label{ipt}
Let $u_{0} \in M^{p,p'},\; p\in (p_{0},2)$ and $N>0$. Then, there exist $\phi_{0} \in L^{2}$ and $\psi_{0} \in M^{p_{0},p_{0}'}$  such that 
\[u_{0}= \phi_{0}+ \psi_{0}, \quad \|\psi_{0}\|_{M^{p_{0},p_{0}'}}\leq C \|u_{0}\|
_{M^{p,p'}}\frac{1}{N}, \quad \|\phi_{0}\|_{L^2}\leq C \|u_{0}\|_{M^{p,p'}}N^{\gamma}. \]
\end{lem} 
\begin{proof}
   Using complex interpolation of modulation spaces, (see \cite[Theorem 6.1(D)]{Feih83}), for $p_{0}<p<2$ and $\theta \in (0,1)$, we have
    \begin{equation*}
[M^{p_{0},p_{0}'},L^{2}]_{\theta}=M^{p,p'},
    \quad\text{with} \quad\theta=\frac{\frac{1}{p_{0}}-\frac{1}{p}}{\frac{1}{p_{0}}-\frac{1}{2}}.
    \end{equation*}
Additionally, invoking \cite[Definition 1.10.1 and Theorem 1.10.3/1]{triebel1995interpol} yields
    \begin{align*}
[M^{p_{0},p_{0}'},L^{2}]_{\theta} \hookrightarrow (M^{p_{0},p_{0}'},L^{2})_{(\theta,\infty)},
    \end{align*}
    where the norm of the space $(M^{p_{0},p_{0}'},L^{2})_{(\theta,\infty)}$ is defined as
    \begin{equation}\label{bound1}
\|u_{0}\|_{(M^{p_{0},p_{0}'},L^{2})_{(\theta,\infty)}} =\sup_{t>0} \inf_{u_{0}=\psi_{0}+\phi_{0}} (t^{-\theta}\|\psi_{0}\|_{M^{p_{0},p_{0}'}}+t^{1-\theta}\|\phi_{0}\|_{L^{2}}), \quad\text{for}\; \psi_{0}\in M^{p_{0},p_{0}'},  \phi_{0}\in L^{2}.
    \end{equation}
    Given $N>0,$ taking $
      t=N ^{-\frac{1}{\theta}} $ in \eqref{bound1}, we obtain
    \begin{equation}\label{bound2}
       N\|\psi_{0}\|_{M^{p_{0},p_{0}'}}+ N^{-\gamma}\|\phi_{0}\|_{L^2} \lesssim \|u_{0}\|_{M^{p,p'}}, \quad\gamma=\frac{1-\theta}{\theta}=\frac{\frac{1}{p}-\frac{1}{2}}{\frac{1}{p_{0}}-\frac{1}{p}}.
    \end{equation}
Rearranging \eqref{bound2}, we get our result.
\end{proof}
{\bf Acknowledgements :} The second author acknowledges the financial support from the University
Grants Commission (UGC), India (file number 201610135365) for pursuing the PhD program. The second author would like to thank Prof. Ryosuke Hyakuna for clarifying key aspects of their research. 
\vspace{-0.6cm}
\bibliographystyle{plain}
\bibliography{DDV18.bib}
\end{document}